\documentclass[10pt,a4paper]{article}
\usepackage{latexsym}
\usepackage{amsfonts}
\usepackage{amsmath}
\usepackage{amsthm}
\usepackage{amssymb}
\usepackage{hyperref}
\usepackage{graphicx}
\usepackage{subfigure}
\usepackage{balance}
\usepackage{multicol}
\usepackage{float}
\usepackage{yhmath}
\usepackage{epstopdf}
\usepackage{color}
\usepackage{etoolbox}
\AtBeginEnvironment{cases}{\linespread{1.2}\selectfont}

\def\p{\partial}
\def\ve{\varepsilon}
\def\f{\frac}
\def\na{\nabla}
\def\la{\lambda}
\def\La{\Lambda}
\def\al{\alpha}

\def\vp{\varphi}
\def\O{\Omega}
\def\o{\omega}

\def\th{\theta}

\def\g{\gamma}
\def\G{\Gamma}
\def\si{\sigma}
\def\Si{\Sigma}

\def\dl{\delta}

\def\ds{\displaystyle}

\def\i{\infty}

\def\no{\nonumber}

\def\beq{\begin{equation}}
\def\eeq{\end{equation}}
\def\ben{\begin{eqnarray}}
\def\een{\end{eqnarray}}
\def\bec{\begin{cases}}
\def\eec{\end{cases}}

\newcommand{\bR}{{\mathbb R}}

\newcommand{\vF}{{\mathcal F}}
\newcommand{\vG}{{\mathcal G}}

\begin{document}
\newtheorem{theorem}{Theorem}
\newtheorem{lemma}{Lemma}
\renewcommand{\thelemma}
{\arabic{section}.\arabic{lemma}}
\newtheorem{corollary}[lemma]{Corollary}
\newtheorem{remark}{Remark}
\renewcommand{\theremark}
{\arabic{section}.\arabic{remark}}
\renewcommand{\theequation}
{\arabic{section}.\arabic{equation}}
\renewcommand{\thetheorem}
{\arabic{section}.\arabic{theorem}}
\makeatletter
\@addtoreset{equation}{section} \makeatother \makeatletter
\makeatletter
\@addtoreset{lemma}{section} \makeatother \makeatletter
\makeatletter
\@addtoreset{remark}{section} \makeatother \makeatletter
\makeatletter \@addtoreset{theorem}{section} \makeatother
\makeatletter
\makeatother
\title{\bf {Formation and construction of a multidimensional shock wave for the first order  hyperbolic
conservation law with smooth initial data}}
\author{Yin Huicheng$^{1,*}$, \quad Zhu Lu$^{2,}$
\footnote{Yin Huicheng (huicheng@nju.edu.cn, 05407@njnu.edu.cn)
and Zhu Lu (zhulu@hhu.edu.cn) are
supported by the NSFC (No.11731007, No.12001162).}\vspace{0.5cm}\\
1.  School of Mathematical Sciences and Institute of Mathematical Sciences,\\
Nanjing Normal University, Nanjing, 210023, China.
\\
\vspace{0.5cm}
2. College of Science, Hohai University, Nanjing, 210098, China.}
\date{}
\maketitle

\begin{abstract}
In this paper, the problem on formation and construction of a multidimensional
shock wave  is studied for the first order conservation law $\p_t u+\p_x F(u)+\p_y G(u)=0$ with smooth
initial data $u_0(x,y)$. It is well-known that the smooth solution $u$ will blow up on the time $T^*=-\f{1}{\min{H(\xi,\eta)}}$
when $\min{H(\xi,\eta})<0$ holds for $H(\xi,\eta)=\p_{\xi}(F'(u_0(\xi,\eta)))+\p_{\eta}(G'(u_0(\xi,\eta)))$,
more precisely, only the first order derivatives $\na_{t,x,y}u$ blow up  on $t=T^*$ meanwhile $u$ itself is still
continuous until $t=T^*$. Under the generic nondegenerate condition of $H(\xi,\eta)$, we construct a local weak
entropy solution $u$ for $t\ge T^*$ which is not uniformly Lipschitz continuous on two sides of
a shock surface $\Sigma$. The strength of the constructed shock is zero on the initial
blowup curve $\G$ and then gradually increases for $t>T^*$. Additionally,  in the neighbourhood of $\G$,
some detailed and precise descriptions on the singularities of solution $u$ are given.

\end{abstract}

\begin{quote} {\bf Keywords:} Hyperbolic conservation law, multidimensional shock wave, generic nondegenerate condition,
entropy condition, Rankine-Hugoniot condition.
\end{quote}
\vskip 0.2 true cm

\begin{quote} {\bf Mathematical Subject Classification 2000:} 35L05, 35L72 \end{quote}


\section{Introduction}

\subsection{Setting of the problem and statement of the main result}
In this paper, we shall study the  problem of a multidimensional shock formation for the
following first order 2D conservation law
\beq\label{2.1}
\left\{
  \begin{array}{l}
    \p_t u+\p_x F(u)+\p_y G(u)=0,  \\
    u(0,x,y)=u_0(x,y),
  \end{array}
\right.
\eeq
where $(t,x,y)\in\bR_+\times\bR^2$, $F(u)$ and $G(u)$ are $C^5$ smooth functions of $u$, $u_0(x,y)\in C^4(\Bbb R^2)$.
Let $f(u)=F'(u)$ and $g(u)=G'(u)$. Define the characteristics $(x(t;\xi,\eta), y(t;\xi,\eta))$ of \eqref{2.1} starting from
the initial point $(\xi,\eta)$ as follows
\beq\label{2.2}
\left\{
\begin{array}{l}
\f{d}{dt}x(t;\xi,\eta)=f(u(t,x(t;\xi,\eta),y(t;\xi,\eta))),\\
\f{d}{dt}y(t;\xi,\eta)=g(u(t,x(t;\xi,\eta),y(t;\xi,\eta))),\\
x(0,\xi,\eta)=\xi,\ y(0,\xi,\eta)=\eta.
\end{array}
\right.
\eeq
As long as the $C^1$ solution $u$ of \eqref{2.1} exists (actually $u\in C^4$ due to $u_0\in C^4$ and $F(u),G(u)\in C^5$),
it then follows from \eqref{2.1} and \eqref{2.2}
that along the characteristics $(t, x(t;\xi,\eta), y(t;\xi,\eta))$,
\beq\label{2.3}
\f{d}{dt}u(t,x(t;\xi,\eta),y(t;\xi,\eta))\equiv0,
\eeq
which derives $u(t,x(t;\xi,\eta),y(t;\xi,\eta))\equiv u_0(\xi,\eta)$. In this case,
we have that from \eqref{2.2}
\beq\label{2.5}
\left\{
\begin{array}{l}
x(t;\xi,\eta)=\xi+tf(u_0(\xi,\eta)),\\
y(t;\xi,\eta)=\eta+tg(u_0(\xi,\eta)).
\end{array}
\right.
\eeq
Obviously, if $\xi=\xi(t,x,y)\in C^1$ and $\eta=\eta(t,x,y)\in C^1$ are obtained from \eqref{2.5},
then $u(t,x,y)=u_0(\xi(t,x,y),\eta(t,x,y))\in C^1$ will be the classical solution of \eqref{2.1}. In fact,
in terms of ${\rm det}(\f{\p(x,y)}{\p(\xi,\eta)})=1+tH(\xi,\eta)$ with
$$
H(\xi,\eta)=\p_{\xi}(f(u_0(\xi,\eta)))+\p_{\eta}(g(u_0(\xi,\eta))),
$$
by the implicit function theorem $(\xi(t,x,y), \eta(t,x,y))\in C^1$ can be achieved for all $t\ge 0$ when $\min{H(\xi,\eta})\ge 0$;
or for $0\le t<T^*$  when $\min{H(\xi,\eta})<0$ and $T^*=-\f{1}{\min{H(\xi,\eta)}}$ since ${\rm det }(\f{\p(x,y)}{\p(\xi,\eta)})>0$ holds
in these two cases.
For $\min{H(\xi,\eta})<0$, it follows from Theorem 3.1 of \cite{Majda-3} that
the $C^1$ solution $u$ of \eqref{2.1} blows up on $T^*=-\f{1}{\min{H(\xi,\eta)}}$, more
precisely, the first order derivatives $\na_{t,x,y}u$ blow up on $t=T^*$ meanwhile $u$ itself is still
continuous until $t=T^*$. In the paper, we are concerned with  the multidimensional
shock formation problem of \eqref{2.1} for $t\ge T^*$ when $\min{H(\xi,\eta})<0$ happens.

For brevity, we denote $\phi(\xi,\eta)=f(u_0(\xi,\eta))$ and $\psi(\xi,\eta)=g(u_0(\xi,\eta))$.
Then
\beq\label{H-0}
H(\xi,\eta)=\p_{\xi}\phi(\xi,\eta)+\p_{\eta}\psi(\xi,\eta).
\eeq
In addition, we pose the following
generic nondegenerate condition on $H(\xi,\eta)$:

{\bf There exists a unique point $(\xi_0,\eta_0)$ such that $H(\xi_0,\eta_0)=\min{H(\xi,\eta)}$, and
$\p_{\xi}H(\xi_0,\eta_0)=\p_{\eta}H(\xi_0,\eta_0)=0$, $(\na^2_{\xi,\eta}H)(\xi_0,\eta_0)$ is symmetric positive.
\qquad \qquad \qquad \qquad \qquad \qquad \qquad \quad (GNC)}

Here we point out that (GNC) has been used  in \cite{A1}
to show the blowup of smooth small data solution
to the second order quasilinear wave equations when the corresponding null conditions
are not fulfilled. For convenience and without loss of generality, we assume that
in (GNC),
\beq\label{Y-0}
\text{$(\xi_0,\eta_0)=(0,0)$ and $\phi(0,0)=\psi(0,0)=0$,}
\eeq
\beq\label{2.6}
H(0,0)=\min_{(\xi,\eta)\in\bR^2}H(\xi,\eta)=-1,
\eeq
\beq\label{2.7}
\left(\begin{array}{cc}\p^2_\xi H & \p^2_{\xi\eta} H\\
\p^2_{\xi\eta} H & \p^2_\eta H
\end{array}\right)(0,0)=\left(\begin{array}{cc}6 & 0\\
0 & 6\end{array}\right)
\eeq
and
\beq\label{Y-1}
\p_\eta\psi(0,0)\geq-\f{1}{2}\geq\p_\xi\phi(0,0).
\eeq

In this case, the unique blowup point $(1,0,0)$ of \eqref{2.1} will appear firstly. Additionally,
from \eqref{Y-0}-\eqref{Y-1}, then there exists a small $\dl>0$ such that for
$(\xi,\eta)\in B_{\dl}=\{(\xi,\eta):\ |\xi|+|\eta|\le \dl\}$,
\beq
\label{2.9}\p_\eta\psi\geq-\f{2}{3},\ H\le-\f{3}{4},\ \textrm{and}\
\phi,\psi,\p_\xi H, \p_\eta H, \p^2_{\xi\eta}H=O(\dl),\ \p_{\xi}^2 H, \p_{\eta}^2H=6+O(\dl).
\eeq

Under conditions \eqref{Y-0}-\eqref{Y-1}, we will prove that for $t\ge T^*=1$, \eqref{2.1}
admits a shock surface $\Sigma$: $x=\vp(t,y)\in C^2$, which  starts
from the space-like blowup curve $\G:\ t=T^*(y),\ x=x^*(y),\ y\in(-\dl,\dl)$
($\G$ will be defined in Lemma \ref{lemma 1.1} below). Denote $u_{-}(t,x,y)$ and $u_{+}(t,x,y)$ by the solution
of \eqref{2.1} on the left ($x<\vp(t)$) and right ($x>\vp(t)$) side of $\Sigma$ respectively (see Figure 1). Then $\vp(t,y)$
satisfies the Rankine-Hugoniot condition on $\Sigma$:
\beq\label{R-H}
\p_t\vp(t,y)[u]-[F(u)]+\p_y\vp(t,y)[G(u)]=0,
\eeq
where $[u]=u_+-u_-$ with $u_\pm=u_\pm(t,x,y)|_{\Sigma}=u_\pm(t,\vp(t,y),y)$.
Note that the formation of shock $\Sigma$ is due to the compression of characteristics, then
the geometric entropy condition on $\Sigma$ is
\beq\label{entropy condition}
\left(1,f(u_+),g(u_+)\right)\cdot\left(-\p_t\vp(t,y),1,-\p_y\vp(t,y)\right)
<0<\left(1,f(u_-),g(u_-)\right)\cdot\left(-\p_t\vp(t,y),1,-\p_y\vp(t,y)\right).
\eeq
Note that $\left(1,f(u_\pm),g(u_\pm)\right)$ is just the tangent direction
of characteristics $\g_\pm$, where $\g_\pm$ stands for the right/left characteristics
of \eqref{2.1} starting from the point $(t,\vp(t,y),y)\in \Sigma$  (see Figure 2).

\begin{figure}[h]
\centering
\includegraphics[scale=0.30]{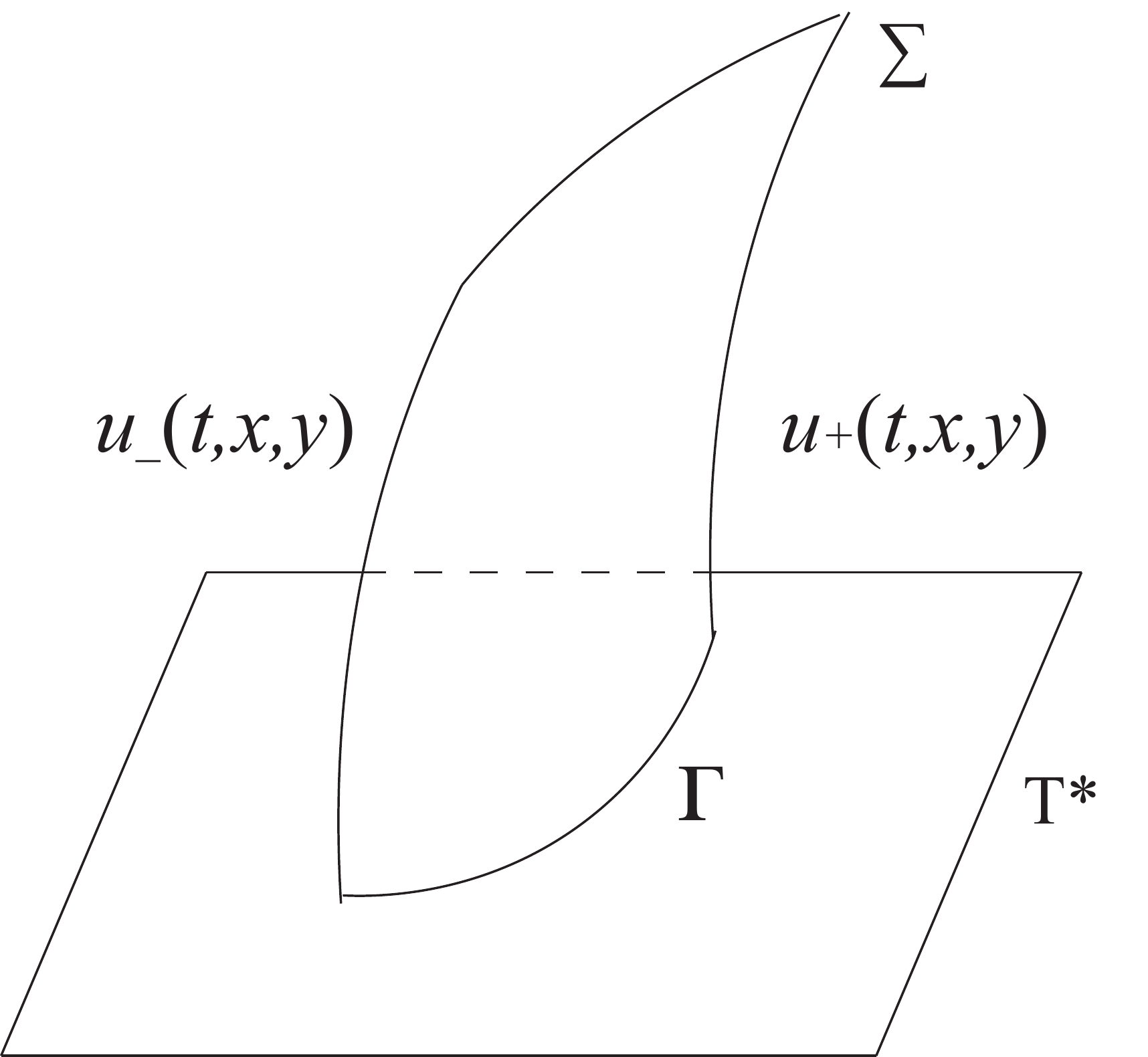}
\end{figure}
\centerline{\bf Figure 1. Shock solution $u=(u_-,u_+)$ and shock surface $\Sigma$ starting from blowup curve $\G$.}

\vskip 0.3 true cm

\begin{figure}[h]
\centering
\includegraphics[scale=0.30]{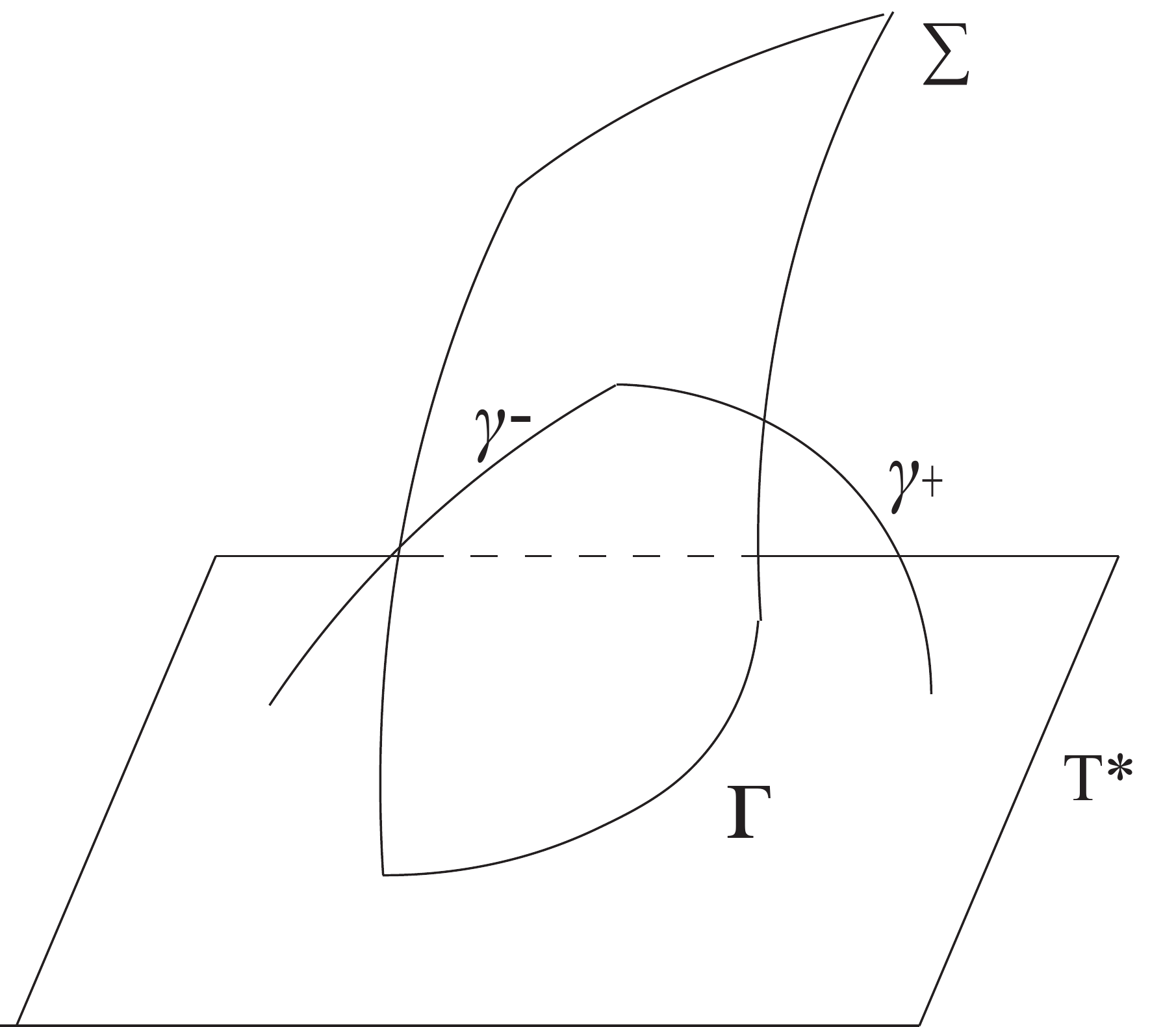}
\end{figure}
\centerline{\bf Figure 2. The characteristics $\g_-$ and $\g_+$.}

\vskip 0.5 true cm

The main results in this paper are
\begin{theorem}\label{theorem}
Under conditions \eqref{Y-0}-\eqref{Y-1}, for small constants $\ve>0$ and $\dl>0$,

\noindent (1) there exist a space-like blowup curve $\G$ for $t\ge 1$:
$t=T^*(y),\ x=x^*(y)$ with $y\in(-\dl,\dl)$ and $(T^*(0),x^*(0))=(1,0)$, and a shock surface $\Sigma:$ $x=\vp(t,y)$
starting from $\G$ in the domain $\O=\{(t,x,y): 1\le t<T^*(y)+\ve, |x|<\dl, |y|<\dl\}$ such that
the Rankine-Hugoniot condition \eqref{R-H} and the entropy condition \eqref{entropy condition} hold.

\noindent (2) $$x=\vp(t,y)\in C^{2}(\O)$$
and $$u\in C^1(\O\setminus\Sigma).$$
\noindent (3) near $\G$ and $t\in (1-\ve, 1+\ve)$,
\ben
&\left|u(t,x,y)-u(T^*(y),x^*(y),y)\right|=O\left(|t-T^*(y)|^\f{1}{2}+|x-x^*(y)-
(\phi^*+\psi^*\f{\p_\eta\phi^*}{\p_\xi\phi^*})(t-T^*(y))|^\f{1}{3}\right),\no\\\label{e0}&&\\
&\left|\nabla_{t,x,y}u(t,x,y)\right|=O\left(\left(|t-T^*(y)|^\f{1}{2}+|x-x^*(y)-
(\phi^*+\psi^*\f{\p_\eta\phi^*}{\p_\xi\phi^*})(t-T^*(y))|^{\f{1}{3}}\right)^{-2}\right),\no\\\label{e1}&&\\
&\left|\p_tu(t,x,y)+\left(\phi^*+\psi^*\f{\p_\eta\phi^*}{\p_\xi\phi^*}\right)\p_xu(t,x,y)\right|
=O\left(\left(|t-T^*(y)|^{\f{1}{2}}+|x-x^*(y)-
(\phi^*+\psi^*\f{\p_\eta\phi^*}{\p_\xi\phi^*})(t-T^*(y))|^{\f{1}{3}}\right)^{-1}\right),\no\\\label{eT}&&
\een
where $\p_t+\left(\phi^*+\psi^*\f{\p_\eta\phi^*}{\p_\xi\phi^*}\right)\p_x$ is the tangent derivative
along the tangent direction $(1,\phi^*+\psi^*\f{\p_\eta\phi^*}{\p_\xi\phi^*},0)$ of $\Si$ at the point
$(T^*(y),x^*(y),y)\in\G$ with the variable $y$ being fixed and $\phi^*=\phi(\xi(t,x,y), \eta(t,x,y))|_{t=T^*(y),x=x^*(y)}$
(the meanings of $\psi^*$, $\p_\eta\phi^*$ and $\p_\xi\phi^*$ are the same as $\phi^*$).
\end{theorem}

\subsection{Remarks and sketch of proof}

\begin{remark}
Theorem \ref{theorem} can be extended into the more general multidimensional first order hyperbolic conservation law
\beq\label{Y-2}
\left\{
  \begin{array}{l}
\p_t u+\ds\sum_{i=1}^n\p_i(F_i(u))=0,\\
u(0,x)=u_0(x),
\end{array}
\right.
\eeq
where $x=(x_1, \cdot\cdot\cdot, x_n)$, $u_0(x)\in C^4(\Bbb R^n)$ ($n\ge 2$), and
$F_i(u)$ $(1\le i\le n)$ is $C^5$ smooth on its argument $u$. If follows from Theorem 3.1 of \cite{Majda-3} that
the solution $u$ of \eqref{Y-2} blows up on $T^*=-\ds\f{1}{\min{H(\xi)}}$
with $H(\xi)=\ds\sum_{i=1}^n\p_{\xi_i}\bigl(F'_i(u_0(\xi))\bigr)$ as long as $\min H(\xi)<0$.
The corresponding generic nondegenerate condition on $H(\xi)$ is as follows

{\bf There exists a unique point $\xi_0\in \Bbb R^n$ such that $H(\xi_0)=\min{H(\xi)}$, and
$\nabla_{\xi}H(\xi_0)=0$, $(\nabla^2_\xi H)(\xi_0)$ is symmetric positive.}

\end{remark}

\begin{remark}
For the 1-D conservation law
\begin{equation}\label{0-1}
\left\{
\begin{aligned}
&\p_tv+\p_xf(v)=0,\\
&v(0,x)=v_0(x),
\end{aligned}
\right.
\end{equation}
where $f(v)\in C^2(\Bbb R)$ and $v_0(x)\in C^1(\Bbb R)$.
It is well-known that the $C^1$ solution $v$ of \eqref{0-1} will blow up at the time $T^*=-\f{1}{\min{g'(x)}}$
with $g(x)=f'(v_0(x))$ and $\min_{x\in\Bbb R}{g'(x)}<0$. If we further assume $g(x)\in L^{\infty}(\Bbb R)\cap C^p(\Bbb R)$
with $p\ge 4$, and pose the following generic nondegenerate condition:
\begin{align}\label{0-2}
\text{\bf There exists a unique point $x_0$ such that $g'(x_0)=\min {g'(x)}<0, g''(x_0)=0, g^{(3)}(x_0)>0$.}
\end{align}
Then by Theorem 2 of \cite{Le94}, a local weak entropy  solution $u$  of \eqref{0-1} together with the shock curve $x=\vp(t)$ starting from
the blowup point $(T^*, x^*=x_0+g(x_0)T^*)$ can be locally obtained. Moreover,
$\vp(t)\in C^p(T^*, T^*+\ve)\cap C^{\f{p}{2}}[T^*, T^*+\ve)$, and if $g(x_0)=0$, then
in some neighbourhood of $(T^*,x^*)$,
\begin{equation}\label{0-4}
\left\{
\begin{aligned}
&|v(t,x)-v(T^*,x^*)|\le C((t-T^*)^3+(x-x^*)^2)^{\f16},\\
&|\p_tv(t,x)|\le C{((t-T^*)^3+(x-x^*)^2)}^{-\f16},\\
&|\p_xv(t,x)|\le C{((t-T^*)^3+(x-x^*)^2)}^{-\f13}.\\
\end{aligned}
\right.
\end{equation}
By comparing \eqref{0-4} with \eqref{e0}-\eqref{e1}, the descriptions on the singularities of $\p_{x}v$ and $\na_{x,y}u$,
$\p_tv$ and $\p_t u+\left(\phi^*+\psi^*\f{\p_\eta\phi^*}{\p_\xi\phi^*}\right)\p_x u$
are analogous. Note that $\p_t v$ is actually the tangent direction along the shock curve at the blowup point $(T^*,x^*)$, which
corresponds to $\p_t u+\left(\phi^*+\psi^*\f{\p_\eta\phi^*}{\p_\xi\phi^*}\right)\p_x u$ in \eqref{eT}.
On the other hand, when condition \eqref{0-2}
is removed, we also study the formation and construction of the shock solution to \eqref{0-1} in \cite{Y-Z}.
\end{remark}

\begin{remark}
We point out that the shock formation problem in Theorem 1.1 is different from the usual Riemann problem
with the discontinuous shock initial data on $t=T^*$. For the latter, the initial data $u_{\pm}(T^*,x,y)$ are
piecewise smooth on the left/right side of $\G$ and are
discontinuous across $\G$ (see Figure 3 and \cite{Majda-1}-\cite{Majda-2}, \cite{G.M}), then the shock solution  $u_{\pm}(t,x,y)$ are also
piecewise smooth on the left/right side of $\Sigma$  for $t\ge T^*$, where $\Sigma$ is the resulting shock surface starting
from $\G$.

\begin{figure}[h]
\centering
\includegraphics[scale=0.39]{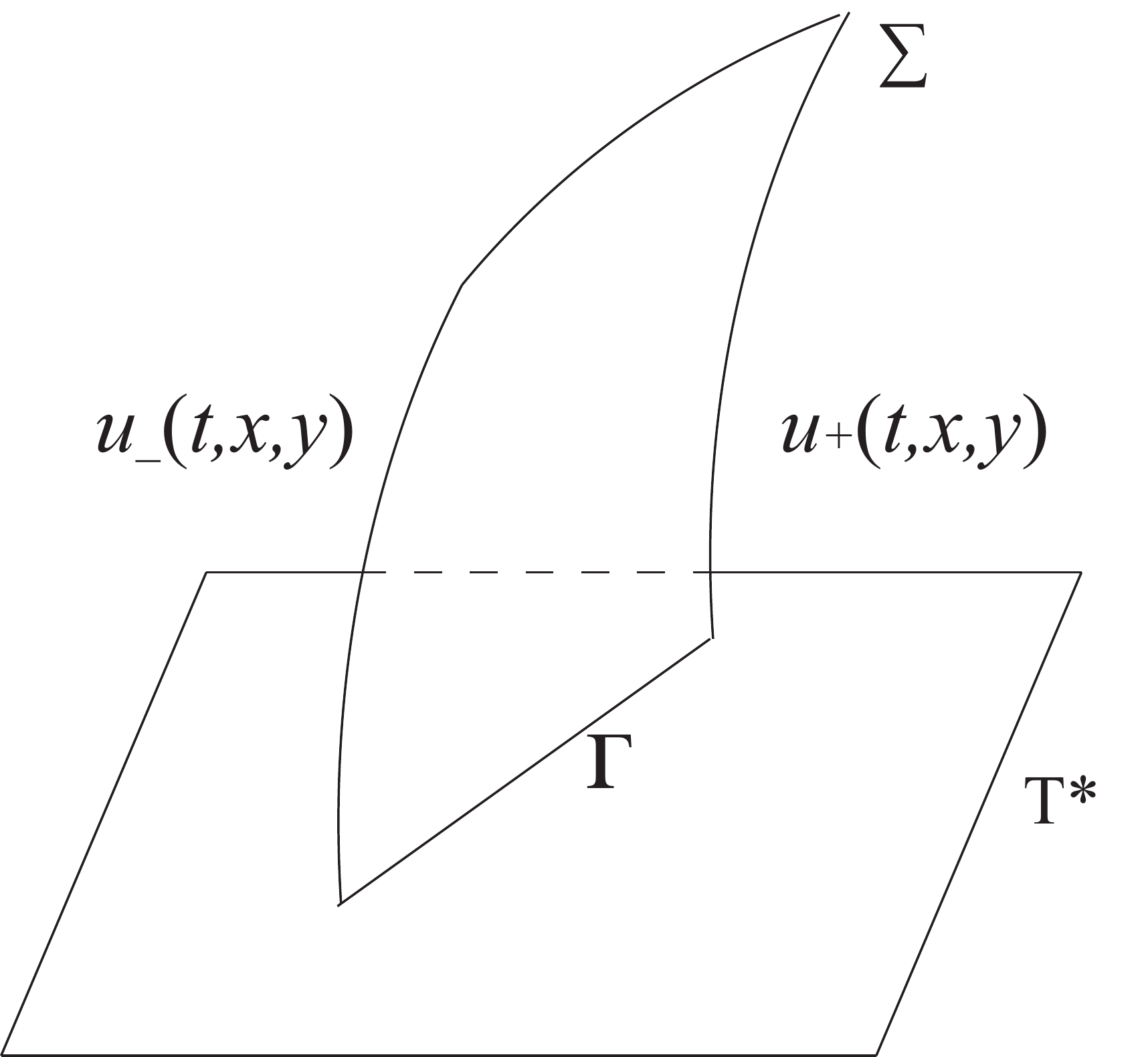}
\end{figure}
\centerline{\bf Figure 3. Riemann problem of \eqref{2.1} on $T^*$.}

\vskip 0.5 true cm

However, in Theorem 1.1, the initial datum $u(T^*, x,y)$ is continuous but is not piecewise smooth
(even not uniformly Lipschitz continuous on two sides of $\G$). More precisely, the strength of the
constructed shock solution $u$ is zero on $\G$ and then gradually increases for $t>T^*$ (see \eqref{e0}-\eqref{e1}
of Theorem 1.1).
\end{remark}

\begin{remark}
Although the global existence of $BV$ solution $u$ of \eqref{2.1} has been early obtained
(see \cite{Hor} or \cite{Smo}), from the viewpoint of understanding
the physical process of the appearance of singularities, it is also  an interesting
problem to give a clear picture on the generation of singularities from a blowup
curve, in particular, that of the singularity of the shock type.
\end{remark}

\begin{remark}
When $u_0(x,y)\in L^{\infty}$, under the entropy condition
\beq\label{H-0}
\p_t\Phi(u)+\p_x F_1(u)+\p_y G_1(u)\le 0\quad \text{in the sense of distribution},
\eeq
where $\Phi$ is any $C^1$ convex function, $F_1(u)=\Phi(u)F'(u)$ and  $G_1(u)=\Phi(u)G'(u)$,
the global existence of a unique weak solution $u$ of \eqref{2.1} has been proved (see Theorem 3.4.3 of \cite{Hor}).
Next we illustrate that \eqref{H-0} means the geometric entropy condition \eqref{entropy condition}
for our shock formation problem.
In fact, by the statements in Pages 44 of \cite{Hor}, near the blowup curve $\G$,
the entropy condition \eqref{H-0} can be described as follows:
Let $\nu=(\nu_0,\nu_1,\nu_2)$ be the normal vector of $\Sigma$, then the function
\beq\label{EC}
[u_-,u_+]\ni \th\mapsto {\rm sgn}(u_+-u_-)\left(F(\th)\nu_1+G(\th)\nu_2\right)
\eeq
lies above the linear interpolation between its values at $u_\pm$.
Without loss of generality, we assume $u_-<u_+$. Then one knows that $\Psi(\th)\triangleq F(\th)\nu_1+G(\th)\nu_2$
is concave in $[u_-,u_+]$. This yields that $\Psi'(\th)=f(\th)\nu_1+g(\th)\nu_2$ is decreasing in $[u_-,u_+]$. Therefore $f(u_+)\nu_1+g(u_+)\nu_2<f(u_-)\nu_1+g(u_-)\nu_2$.
Due to $\nu=(-\p_t\vp,1,-\p_y\vp)$, then we arrive at
\beq\label{EC1}
f(u_+)-\p_y\vp g(u_+)<f(u_-)-\p_y\vp g(u_-).
\eeq
In addition, it follows from \eqref{R-H} that
\beq\label{RH1}
\p_t\vp=\f{\Psi(u_+)-\Psi(u_-)}{u_+-u_-},
\eeq
which derives that there exists a $\th_0\in(u_-,u_+)$ such that $\p_t\vp=\Psi'(\th_0)$.
Note that $\Psi'(\th)$ is decreasing in $[u_-,u_+]$. Then
\beq\label{EC2}
f(u_+)-\p_y\vp g(u_+)<\p_t\vp<f(u_-)-\p_y\vp g(u_-),
\eeq
which is equivalent to \eqref{entropy condition}.
\end{remark}

\begin{remark}
Generally speaking, the descriptions on the singularities of $u$ in \eqref{e0}-\eqref{eT} are optimal. Indeed,
if we consider the following problem
\beq\label{H-1}
\left\{
  \begin{array}{l}
    \p_t v+\p_x (\f{v^2}{2})+\p_y (\f{v^2}{2})=0,  \\
    v(0,x,y)=-x+x^3,
  \end{array}
\right.
\eeq
then as in Sec.10 of \cite{Le94} or Remark 3.1 of \cite{Y-Z}, the regularity
of $\p_x v=O({((t-T^*)^3+(x-x^*)^2)}^{-\f13})$ with $T^*=1$ and $x^*=0$ is optimal.
\end{remark}

Now we briefly mention some interesting  works on the shock formation and construction for the hyperbolic conservation
laws. Under the generic nondegenerate condition of the initial data, for the 1-D scalar conservation law or 1-D $2\times 2$ $p-$
system of
polytropic gases, the authors in \cite{Kong}-\cite{Le94} and \cite{Chen-Dong} obtain the formation and  construction
of a shock wave starting from the blowup point under some variant assumptions; for the 1-D $3\times 3$
strictly hyperbolic conservation laws with the small initial data or the 3-D full compressible Euler equations with
symmetric structure and small perturbation,
the authors in \cite{Chen-Xin-Yin}, \cite{Yin1} and \cite{C-L} also get the formation and  construction
of the resulting shock waves, respectively. From these works, we know that the formation of a shock is caused
by the squeeze of characteristics. On the other hand, in recent years,
the study on the blowup and shock formation of smooth solutions to the
multidimensional hyperbolic conservation laws or the second order potential equations of polytropic
gases have made much progress (see \cite{B-1}-\cite{BSV-3},
\cite{C1}-\cite{0-Speck}, \cite{LS},\cite{Merle}, \cite{MY} and \cite{S2}), which illustrate that the
formation of the multidimensional shock is
due to the compression of the characteristic surfaces. However, the related constructions of multidimensional shock wave
after the blowup of smooth solutions
are not obtained. In the present paper, we are concerned with the construction of a multidimensional shock wave for the
scalar conservation law under the generic nondegenerate condition of the initial data.

In order to prove Theorem 1.1, our focus is to solve the singular and nonlinear first order
partial differential equation \eqref{R-H} of $\vp(t,y)$. The equation \eqref{R-H} is actually equivalent to
$\p_t\vp+h_1(t,y,\vp)\p_y\vp=h_2(t,y,\vp)$, where $h_1(t,y,\vp)=\int_0^1g(\th u_+(t,\vp,y)+(1-\th)u_-(t,\vp,y))d\th$
and $h_2(t,y,\vp)=\int_0^1f(\th u_+(t,\vp,y)+(1-\th)u_-(t,\vp,y))d\th$. Note that
the functions $h_i(t,y,\vp)$ ($i=1,2$) are not Lipschtzian with respect to the variables $(t,y)$ and the unknown function
$\vp$ since the first order derivatives of
$\na_{t,x,y} u_{\pm}(t,x,y)$ admit the strong singularities (see \eqref{e0} and  \eqref{e1}).
To get the uniqueness and
regularities of $(\vp(t,y), u_{\pm}(t,x,y))$, we require to carefully analyze the behaviors
of solution $u$ near the blowup point $(1,0,0)$ and the blowup curve $\G$. By the generic nondegenerate condition (GNC),
at first, we determine the equation and properties of $\G$, meanwhile, a good directional derivative $\p_t+\left(\phi^*
+\psi^*\f{\p_\eta\phi^*}{\p_\xi\phi^*}\right)\p_x$ in \eqref{eT} is found. Subsequently, by careful computation,
the asymptotic behaviors of solution $u$ around $\G$ are
derived and then the existence and regularity of $\vp(t,y)$ are also established. From the result in Theorem 1.1,
we have known a basic fact for problem \eqref{2.1}:  due to the squeeze of characteristic surfaces around the blowup curve,
the shock really appears and develops for $t\ge T^*$ when the initial data satisfy the generic nondegenerate condition.

Our paper is organized as follows. In Section 2, we give some key analysis on the characteristic surface and determine the
blowup curve $\G$ of equation \eqref{2.1} near the blowup point $(1,0,0)$, then complete the construction of the shock surface
$\Sigma$. In Section 3, the behaviors of solution $u$ to problem  \eqref{2.1} around $\G$ are given in details
and then Theorem 1.1 is proved.

\section{Construction of the shock surface $\Sigma$}
From \eqref{2.9} and  by the implicit function theorem, it follows from the second
equation $y=\eta+t\psi(\xi,\eta)$ of \eqref{2.5} that
\beq\label{2.10}
\eta=Y(t,\xi,y)\in C^4,
\eeq
where $t\in[0,\f{4}{3})$, $\xi\in(-\dl,\dl)$ and $y\in(-\dl,\dl)$ for  sufficiently small $\dl>0$.
Meanwhile, it is easy to check that
\beq\label{2.11}
\left(\p_t Y,\p_\xi Y,\p_y Y\right)
=-\f{1}{1+t\p_\eta\psi}\left(\psi,t\p_\xi\psi,-1\right),
\eeq
where $1+t\p_\eta\psi\ge\f{1}{8}$ for $t\in[0,\f{7}{6}]$.
Note that all the derivatives of $Y$ in \eqref{2.11} are uniformly bounded and $\p_tY=O(\dl)$ holds
when $t\in[0,\f{7}{6}]$, $\xi\in(-\dl,\dl)$ and $y\in(-\dl,\dl)$.

At first, we study the property of  the  surface $\Si_0$ generated by
\beq\label{H-5}
D(t,\xi,Y(t,\xi,y))\equiv 1+tH(\xi,Y(t,\xi,y))=0,
\eeq
where $t\ge T^*=1$. Here we point out that the variable $\xi$ in \eqref{H-5} can be actually
expressed a function of $(t,x,y)$ from \eqref{2.5}
when $(t,x,y)\in \Sigma_0$, but $\xi$ has two different expressions for the left part and
the right part of $\Sigma_0$ (see Figure 4).

\begin{lemma}\label{lemma 1.1}
The surface $\Si_0$ is of cusp type, whose edge $\G$ (blowup curve) is space-like.
\end{lemma}

\begin{proof}
For each fixed $y\in(-\dl,\dl)$ and the function $Y(t,\xi,y)$ defined in \eqref{2.10}, let
\beq\label{2.12}
D(t,\xi,Y(t,\xi,y))=0.
\eeq
Then one has
\beq\label{2.13}
\p_t\left(D(t,\xi,Y(t,\xi,y))\right)=H(\xi,Y(t,\xi,y))-t \p_\eta H(\xi,Y(t,\xi,y))\f{\psi}{1+t\p_\eta\psi}.
\eeq
This together with \eqref{2.9} yields that for $t\in[0,\f{7}{6}]$ and $\xi,\ y\in(-\dl,\dl)$,
\beq\label{2.14}
\p_t\left(D(t,\xi,Y(t,\xi,y))\right)\leq
-\f{3}{4}+\f{7}{6}\cdot O(\dl)<0.
\eeq
Hence, for $\xi\in(-\dl,\dl)$ and $y\in(-\dl,\dl)$,
there exists a unique function $t=t(\xi,y)\in C^1$ (the equation of $\Sigma_0$) satisfying $D(t(\xi,y),\xi,y)\equiv0$
by the implicit function theorem.

\begin{figure}[h]
\centering
\includegraphics[scale=0.4]{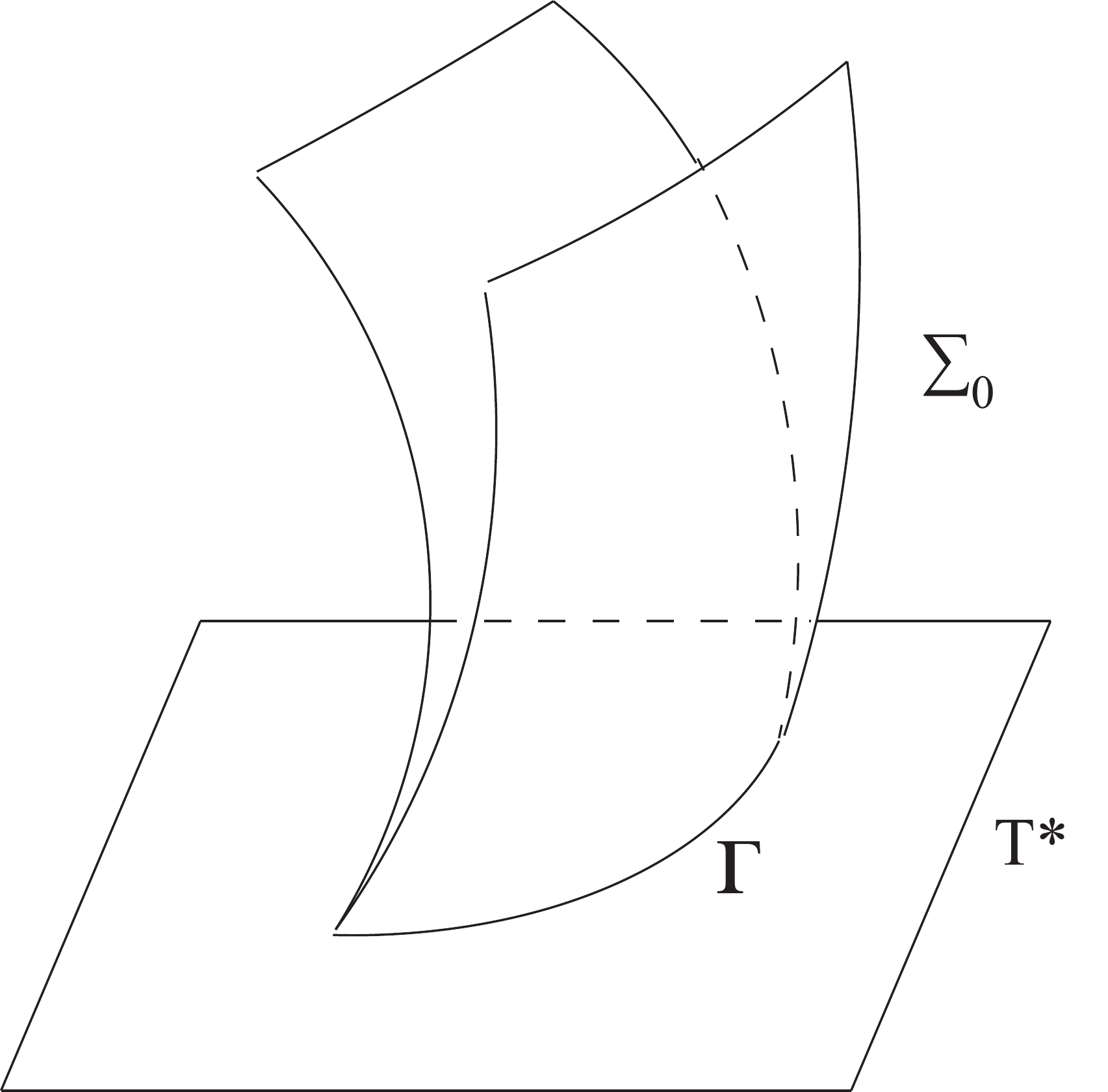}
\end{figure}
\centerline{\bf Figure 4. Surface $\Si_0$ and its edge $\G$.}

\vskip 0.4 true cm

Next we discuss the blowup time $T^*(y)=\min_{\xi} t(\xi,y)$ of solution $u$ to \eqref{2.1}
for fixed $y\in(-\dl,\dl)$. Note that
\beq\label{2.15}
\p_\xi t H(\xi,Y)+t \left(\p_\xi H+\p_\eta H\cdot(\p_\xi Y+\p_t Y\p_\xi t)\right)=0.
\eeq
If taking $t_\xi=0$ in \eqref{2.15}, we then obtain that the related $\xi$ should satisfy at $t=T^*(y)$,
\beq\label{2.16}
D'(t,\xi,Y(t,\xi,y))\triangleq \p_\xi H(\xi,Y(t,\xi,y))+\p_\eta H(\xi,Y(t,\xi,y))\cdot \p_\xi Y(t,\xi,y)=0.
\eeq
By \eqref{2.7}, \eqref{2.9} and direct computation, the following Jacobian determinant holds
that for $\xi,\ y\in(-\dl,\dl)$ and small $\dl>0$,
\begin{align*}
&\left|
        \begin{array}{cc}
          \p_t(D(t,\xi,Y(t,\xi,y))) & \p_\xi(D(t,\xi,Y(t,\xi,y))) \\
         \p_t(D'(t,\xi,Y(t,\xi,y))) & \p_\xi(D'(t,\xi,Y(t,\xi,y))) \\
        \end{array}
\right|\no\\
&=\left|
        \begin{array}{cc}
          H+t \p_\eta H \p_t Y & t D'(t,\xi,Y(t,\xi,\eta)) \\
         \p^2_{\xi\eta} H\p_t Y+\p^2_\eta H\p_t Y \p_\xi Y+\p_{\eta} H\p^2_{t\xi} Y &
         \p^2_\xi H+2\p^2_{\xi\eta} H\p_\xi Y+\p^2_\eta H
          (\p_\xi Y)^2+\p_{\eta} H\p^2_\xi Y\\
        \end{array}
\right|\no\\
&=\left(H+t \p_\eta H \p_t Y\right)\left(\p^2_\xi H+2\p^2_{\xi\eta} H\p_\xi Y+\p^2_\eta H
(\p_\xi Y)^2+\p_{\eta} H\p^2_\xi Y\right)\no\\
\end{align*}

\ben\label{2.17}
&=&H\left(\p^2_\xi H+\p^2_\eta H\left(\p_\xi Y\right)^2\right)+O(\dl)\no\\
&=&-6-6\th_0+O(\dl)<0,
\een
where $\th_0=\left(\f{\p_\xi \psi(0,0)}{\p_\xi \phi(0,0)}\right)^2$.
Thus by the implicit function theorem, we obtain the solution $(\xi,\eta)=(\Xi^*(y),Y^*(y))$
from equations \eqref{2.12} and \eqref{2.16}.
Furthermore,  we can get the blowup curve $\G$ with the parameter $y$ as follows
\beq\label{2.18-0}
\ t=T^*(y)=-\f{1}{H(\Xi^*(y),Y^*(y))},\ x=\Xi^*(y)+T^*(y)\phi(\Xi^*(y),Y^*(y)),
\eeq
which comes from the first equation of the characteristics \eqref{2.5} starting from  the initial
point $(t,\xi,\eta)=(0,\Xi^*(y),Y^*(y))$.

Below we study  the properties of surface $\Si_0$.
For each $y\in(-\dl,\dl)$, we have the Taylor's expansion
of $H(\xi,Y(t,\xi,y))$ at the point $(t,\xi)=(T^*(y),\Xi^*(y))$:
\ben\label{2.18}
&&H(\xi,Y(t,\xi,y))\no\\
&&=H^*+\p_\eta H^*\p_t Y^*(t-T^*(y))+\f{1}{2}\left[\p^2_\xi H^*+2\p^2_{\xi\eta}H^*\p_\xi Y^*+\p^2_\eta H^*
(\p_\xi Y^*)^2+\p_{\eta} H^*\p^2_\xi Y^*\right](\xi-\Xi^*(y))^2\no\\ 
&&\quad +\left[\p^2_{\xi\eta} H^*\p_t Y^*+\p^2_\eta H^*\p_\xi Y^*\p_t Y^*+\p_{\eta} H^*\p^2_{t\xi}Y^*\right]
(t-T^*(y))(\xi-\Xi^*(y))\no\\
&&\quad +a^*(y)(\xi-\Xi^*(y))^3+O((t-T^*(y))^2+(t-T^*(y))|\xi-\Xi^*(y)|^2+|\xi-\Xi^*(y)|^4),
\een
where
\beq
H^*=H(\xi,Y(t,\xi,y))|_{(t,\xi)=(T^*(y),\Xi^*(y))},\ \p_\eta H^*=(\p_\eta H)(\xi,Y(t,\xi,y))|_{(t,\xi)=(T^*(y),\Xi^*(y))},\ ...,
\eeq
and
\ben
a^*(y)&=&\f{1}{6}\left[\p^3_\xi H^*+3\p^2_\xi\p_\eta H^*\p_\xi Y^*
+3\p_\xi\p^2_\eta H^*(\p_\xi Y^*)^2+2\p^2_{\xi\eta} H^*\p_\xi^2 Y^*
+\p^3_\eta H^*(\p_\xi Y^*)^3\right.\no\\
&&\left.+3\p^2_\eta H^*\p_\xi Y^* \p^2_\xi Y^*
+\p_\eta H^* \p^3_\xi Y^*\right].
\een
From now on, the notation $F^*=F|_{(t,\xi,\eta)=(T^*(y), \Xi^*(y),Y^*(y))}$
stands for the value of function $F$
at the point $(t,\xi,\eta)=(T^*(y), \Xi^*(y),Y^*(y))$.
In addition,  by $T^*(y)=-\f{1}{H^*}$, we have
\ben\label{2.19}
&&D(t,\xi,Y(t,\xi,y))\no\\
&&=1+T^*(y)H(\xi,Y(t,\xi,y))+(t-T^*(y))H(\xi,Y(t,\xi,y))\no\\
&&=-D_0(y)(t-T^*(y))+D_1(y)(\xi-\Xi^*(y))^2+D_2(y)(t-T^*(y))(\xi-\Xi^*(y))+D_3(y)(\xi-\Xi^*(y))^3\no\\
&&\quad +O((t-T^*(y))(\xi-\Xi^*(y))^2+(t-T^*(y))^2+|\xi-\Xi^*(y)|^3),
\een
where
\ben
\label{2.20.1}D_0(y)&=&-\left(H^*+T^*(y)\p_\eta H^* \p_t Y^*\right),\no\\
\label{2.20.2}D_1(y)&=&\f{T^*(y)}{2}\left[\p^2_\xi H^*+2\p^2_{\xi\eta}H^*\p_\xi Y^*+\p^2_\eta H^*
(\p_\xi Y^*)^2+\p_{\eta} H^*\p^2_\xi Y^*\right],\no\\
\label{2.20.3}D_2(y)&=&T^*(y)\left[\p^2_{\xi\eta} H^*\p_t Y^*+\p^2_\eta H^*\p_\xi Y^*\p_t  Y^*
+\p_{\eta} H^*\p^2_{t\xi} Y^*\right],\no\\
\label{2.20.4}D_3(y)&=&a^*(y)T^*(y).
\een
Note that $H^*=-1+o(\dl)$, $T^*(y)=-\f{1}{H^*}=1+O(\dl)$, $\p^2_\xi H^*=6+O(\dl)$,
$\p^2_\eta H^*=6+O(\dl)$ and $\p_\eta H^*, \p_t Y^*, \p^2_{\xi\eta} H^*=O(\dl)$ for $y\in(-\dl,\dl)$ and $\dl>0$
sufficiently small.

To obtain the expansion of $\xi-\Xi^*(y)$ in \eqref{2.19},
we consider the equation
\beq\label{2.21}
h(s,x)\triangleq-s+ax^2+bsx+cx^3+O(s^2+s|x|^2+x^4)=0,
\eeq
where $a>0$, $b,c\in{\mathbb R}$ and $h\in C^4$. Taking $\o=s^\f{1}{2}$ and $\si=\f{x}{\o}$, \eqref{2.21} is equivalent to
\beq
\tilde{h}(\o,\si)\triangleq\f{h(\o^2,\o\si)}{\o^2}=-1+a\si^2+b\o\si+c\o\si^3+O(\o^2)=0,
\eeq
where $\tilde h(\o,\si)\in C^2$. Since $\tilde h(0,\pm\f{1}{\sqrt{a}})=0$ and $\f{\p\tilde h}{\p\si}=\pm2\sqrt{a}\neq0$,
thus by implicit function theorem we obtain
\beq
\si_\pm(\o)=\pm\f{1}{\sqrt a} -\f{ab+c}{2a^2}\o+O(\o^2)\in C^2
\eeq
satisfying $\tilde h(\o,\si_\pm(\o))\equiv 0$ for $\o$ near $0$. Then \eqref{2.21}
has the two real roots $x_\pm(s)\in C^{1/2}$ for $s\in[0,\ve)$ satisfying
\beq\label{2.22}
x_\pm(s)=\pm\f{1}{\sqrt a} s^\f{1}{2}-\f{ab+c}{2a^2}s+O(s^\f{3}{2}).
\eeq

This yields that for $(t,\xi,y)\in\Si_0$,
\beq\label{2.23}
\xi=\Xi_\pm(t,y)=\Xi^*(y)\pm A_1(y)(t-T^*(y))^{\f{1}{2}}+A_2(y)(t-T^*(y))+O((t-T^*(y))^\f{3}{2}),
\eeq
where
\ben
\label{2.24.1}A_1(y)&=&\sqrt{\f{D_0(y)}{D_1(y)}}=\f{1}{\sqrt{3+3\th_0}}+O(\dl),\no\\
\label{2.24.1}A_2(y)&=&-\f{D_1(y)D_2(y)+D_0(y)D_3(y)}{2\left(D_1(y)\right)^2}.
\een
In addition, by  recalling $x^*(y)=\Xi^*(y)+T^*(y)\phi^*$ and $T^*(y)\left(\p_\xi\phi^*+\p_\eta\phi^*\p_\xi  Y^*\right)=-1$,
for $(t,\xi,y)\in\Si_0$ with $t\ge T^*(y)$ and $y\in(-\dl,\dl)$, we can arrive at
\ben\label{2.25}
&&x_\pm(t,y)\no\\
&&\equiv \Xi_\pm(t,y)+T^*(y)\phi(\Xi_\pm(t,y),Y(t,\Xi_\pm(t,y),y))
+(t-T^*(y))\phi(\Xi_\pm(t,y),Y(t,\Xi_\pm(t,y),y))\no\\
&&=\Xi_\pm(t,y)+T^*(y)\phi^*+T^*(y)\p_\eta \phi^*\p_t Y^*(t-T^*(y))
+T^*(y)\left(\p_\xi\phi^*+\p_\eta \phi^* \p_\xi Y^*\right)(\Xi_\pm(t,y)-\Xi^*(y))\no\\
&&\quad +\f{T^*(y)}{2}
\left.\left[\p^2_\xi\left(\phi(\xi,Y(t,\xi,y))\right)\right]
\right|_{t=T^*(y),\xi=\Xi^*(y)}(\Xi_\pm(t,y)-\Xi^*(y))^2\no\\
&&\quad +T^*(y)
\left.\left[\p^2_{t\xi}\left(\phi(\xi,Y(t,\xi,y))\right)\right]
\right|_{t=T^*(y),\xi=\Xi^*(y)}
(t-T^*(y))(\Xi_\pm(t,y)-\Xi^*(y))\no\\
&&\quad +\f{T^*(y)}{6}\left.\left[\p^3_{\xi}\left(\phi(\xi,Y(t,\xi,y))\right)\right]
\right|_{t=T^*(y),\xi=\Xi^*(y)}(\Xi_\pm(t,y)-\Xi^*(y))^3\no\\
&&\quad +(t-T^*(y))\phi^*+(t-T^*(y))(\p_\xi \phi^*+\p_\eta \phi^* \p_\xi Y^*)(\Xi_\pm(t,y)-\Xi^*(y))\no\\
&&\quad +O((t-T^*(y))^2+(t-T^*(y))(\Xi_\pm(t,y)-\Xi^*(y))^2+|\Xi_\pm(t,y)-\Xi^*(y)|^4)\no\\
&&=x^*(y)+(\phi^*+\psi^*\f{\p_\eta\phi^*}{\p_\xi\phi^*})(t-T^*(y))
\mp b^*(y)(t-T^*(y))^\f{3}{2}+O((t-T^*(y))^2),
\een
where
\ben\label{2.26}
b^*(y)&=&- A_1(y)T^*(y)\left.\left[\p^2_{t\xi}\left(\phi(\xi,Y(t,\xi,y))\right)\right]
\right|_{t=T^*(y),\xi=\Xi^*(y)}-A_1(y)H^*\no\\
&&\quad -\f{T^*(y)}{6}A_1^3(y)\left.\left[\p^3_{\xi}\left(\phi(\xi,Y(t,\xi,y))\right)\right]
\right|_{t=T^*(y),\xi=\Xi^*(y)}.
\een
Note that
\beq\label{2.27}
\p_\xi\left(\phi(\xi,Y(t,\xi,y))\right)=
\f{\p_\xi\phi}{1+t\p_\eta\psi}=\f{1}{t}-\f{1+t H}{t(1+t\p_\eta\psi)}
\eeq
and
\beq\label{2.28}
\p^2_\xi\left(\phi(\xi,Y(t,\xi,y))\right)
=\f{1+t H}{(1+t\p_\eta \psi)^2}\p_\xi(\p_\eta \psi(\xi,Y(t,\xi,y)))
-\f{\p_\xi H+\p_\eta H \p_\xi Y}{1+t\p_\eta \psi}.
\eeq
Due to $1+t H\equiv0$ and $\p_\xi H+\p_\eta H \p_\xi Y\equiv0$ for $(t,\xi,y)\in \G$,
we then have that from \eqref{2.28},
\beq\label{2.29}
\left.\left[\p^2_\xi\left(\phi(\xi,Y(t,\xi,y))\right)\right]\right|_{t=T^*(y),\xi=\Xi^*(y)}\equiv0.
\eeq
On the other hand,
\ben\label{2.29-1}
&&\p^3_\xi\left(\phi(\xi,Y(t,\xi,y))\right)\no\\
&=&(1+t H)\p_\xi
\left[\f{1}{(1+t\p_\eta \psi)^2}\p_\xi(\p_\eta \psi(\xi,Y(t,\xi,y)))\right]
+\f{t(\p_\xi H+\p_\eta H \p_\xi Y)}{(1+t\p_\eta \psi)^2}\p_\xi(\p_\eta \psi(\xi,Y(t,\xi,y)))\no\\
&&-(\p_\xi H+\p_\eta H \p_\xi Y)\p_\xi\left(\f{1}{1+t\p_\eta \psi}\right)
-\f{\p^2_\xi H+2\p^2_{\xi\eta} H \p_\xi Y+\p^2_\eta H\left(\p_\xi Y\right)^2
+\p_\eta H\p^2_\xi Y}{1+t\p_\eta \psi},
\een
which derives
\beq\label{2.29-2}
\left.\left[\p^3_\xi\left(\phi(\xi,Y(t,\xi,y))\right)\right]
\right|_{t=T^*(y),\xi=\Xi^*(y)}=-\f{6+6\th_0}{\p_\xi \phi(0,0)}+O(\dl).
\eeq

In addition,
\ben
\label{2.29-3}&&\left.\left[\p^2_{t\xi}\left(\phi(\xi,Y(t,\xi,y))\right)\right]\right|_{t=T^*,\xi=\Xi^*}\no\\
&=&\p^2_{\xi\eta}\phi^* \p_t Y^*
+\p^2_\eta \phi^* \p_t Y^*\p_\xi Y^*+\p_\eta\phi^* \p^2_{t\xi}Y^*\no\\
&=&-\f{\p_\eta \phi^*\p_\xi \psi^*}{(1+T^*\p_\eta \psi^*)^2}+O(\dl)
=-\f{\p_\xi \phi^*\p_\eta \psi^*}{(-T^*\p_\xi \phi^*)^2}+O(\dl)\no\\
&=&-\f{\p_\eta \psi(0,0)}{\p_\xi \phi(0,0)}+O(\dl).
\een
Therefore, we have that for small  $\dl>0$,
\beq\label{2.29-4}
b^*(y)=-\f{2}{3\p_\xi \phi(0,0)\sqrt{3+3\th_0}}+O(\dl)>0.
\eeq
This means that $\Si_0$ is of a cusp-typed surface. Additionally, the edge $\G$ of $\Si_0$
is space-like in terms of \eqref{2.18-0}.
\end{proof}


\begin{remark}
We now explain the geometric meaning of the coefficient $\phi^*+\psi^*\f{\p_\eta \phi^*}{\p_\xi \phi^*}$
of $t-T^*(y)$ in \eqref{e0}-\eqref{eT}. Note that the surface
$\Si_0$ is determined by the parametric equations of $(\xi,\eta)$
\beq\label{2.30}
\left\{\begin{array}{l}
         t=-\f{1}{H},\\
         x=\xi-\f{1}{H}\phi(\xi,\eta),\\
         y=\eta-\f{1}{H}\psi(\xi,\eta).
       \end{array}
\right.
\eeq
Then its normal vector is
\ben\label{2.31}
&&(\p_\xi t,\p_\xi x,\p_\xi y)\times(\p_\eta t,\p_\eta x,\p_\eta y)\\
&&=\f{1}{H^3}\left(\p_\xi H(\phi\p_\xi \phi+\psi\p_\eta \phi)
+\p_\eta H(\phi\p_\xi \psi+\psi\p_\eta \psi),
-(\p_\xi H\p_\xi \phi+\p_\eta H\p_\xi \psi),-(\p_\xi H\p_\eta \phi+\p_\eta H\p_\eta \psi)\right).\no
\een
Thus for fixed $y\in(-\dl,\dl)$, one has a tangent direct of $\Si_0$ as follows
\beq\label{2.32}
\f{dx}{dt}=\f{\p_\xi H(\phi\p_\xi \phi+\psi\p_\eta \phi)
+\p_\eta H(\phi\p_\xi \psi+\psi\p_\eta \psi)}{\p_\xi H\p_\xi \phi+\p_\eta H\p_\xi \psi}
=\phi+\f{\psi(\p_\xi H\p_\eta \phi+\p_\eta H\p_\eta \psi)}{\p_\xi H\p_\xi \phi+\p_\eta H\p_\xi \psi}.
 \eeq
Due to $\p_\xi \phi<0$, then
\beq\label{2.33}
\f{\p_\xi H\p_\eta \phi+\p_\eta H\p_\eta \psi}{\p_\xi H\p_\xi \phi+\p_\eta H\p_\xi \psi}
=\f{\p_\xi \phi(\p_\xi H\p_\eta \phi+\p_\eta H\p_\eta \psi)}{\p_\xi \phi(\p_\xi H\p_\xi \phi+\p_\eta H\p_\xi \psi)}
=\f{\p_\eta \phi(\p_\xi H\p_\xi \phi+\p_\eta H\p_\xi \psi)}
{\p_\xi \phi(\p_\xi H\p_\xi \phi+\p_\eta H\p_\xi \psi)}
=\f{\p_\eta \phi}{\p_\xi \phi}.
\eeq
Substituting this into \eqref{2.32} yields
\beq\label{2.34}
\f{dx}{dt}=\phi+\psi\f{\p_\eta \phi}{\p_\xi \phi},
\eeq
which means that $\phi^*+\psi^*\f{\p_\eta \phi^*}{\p_\xi \phi^*}$ is just the tangent direction of $\G$ at
the point $(T^*(y),x^*(y),y)$ for the fixed $y\in(-\dl,\dl)$.
\end{remark}

We rewrite the first equation of the characteristics \eqref{2.5} starting from the point $(T^*(y), \Xi^*(y), Y^*(y))$ as
\beq\label{2.36}
x-x^*=(\xi-\Xi^*(y))+t\phi(\xi,Y(t,\xi,y))-T^*(y)\phi(\Xi^*(y),Y^*(y)).
\eeq
For the fixed $y\in(-\dl,\dl)$ and $t>T^*(y)$, we introduce  the following transformation
\beq\label{2.37}
s=(t-T^*(y))^\f{1}{2},\ \mu=\f{\xi-\Xi^*(y)}{s},\
\la=\f{x-x^*(y)-(\phi^*+\psi^*\f{\p_\eta \phi^*}{\p_\xi \phi^*})(t-T^*(y))}{s^3}.
\eeq
\begin{lemma}\label{lemma 1.2}
For fixed $y\in(-\dl,\dl)$, $t>T^*(y)$, and small $\ve>0$, when $s\in[0,\ve)$ and $|\la|<\ve$,
there exist two real roots for the characteristics equation \eqref{2.5} with  the following expansions
\ben
\label{2.39}\xi_+(t,x,y)&=&\Xi^*(y)+s\left(\sqrt{\f{c_1(y)}{c_2(y)}}+\f{\la}{2c_1(y)}\right)+O(s^2+s\la^2),\\
\label{2.40}\xi_-(t,x,y)&=&\Xi^*(y)+s\left(-\sqrt{\f{c_1(y)}{c_2(y)}}+\f{\la}{2c_1(y)}\right)+O(s^2+s\la^2),
\een
where $c_1(y)$ and $c_2(y)$ are some suitable  positive functions. Furthermore, $s\to
s^\al\left(\xi_\pm-\Xi^*(y)\right)$ is of $C^{2+\al}$ for $\al=-1,0,1$.
\end{lemma}
\begin{proof}
At first, we have that  for fixed $y\in(-\dl,\dl)$,
\ben\label{2.41}
&&t\phi(\xi,Y(t,\xi,y))-T^*(y)\phi(\Xi^*(y),Y^*(y))\no\\
&=&(\phi^*+T^*(y)\p_\eta \phi^*\p_t  Y^*)(t-T^*(y))+T^*(y)(\p_\xi \phi^*+\p_\eta \phi^*\p_\xi  Y^*)(\xi-\Xi^*(y))\no\\
&& +\f{T^*(y)}{2}\left[\p^2_\xi
\left(\phi(\xi,Y(t,\xi,y))\right)\right]_{t=T^*(y),\xi=\Xi^*(y)}(\xi-\Xi^*(y))^2\no\\
&& +\left(T^*(y)\left[\p^2_{t\xi}\left(\phi(\xi,Y(t,\xi,y))\right)\right]_{t=T^*(y),\xi=\Xi^*(y)}
+(\p_\xi \phi^*+\p_\eta \phi^* \p_\xi Y^*)\right)(t-T^*(y))(\xi-\Xi^*(y))\no\\
&& +\f{T^*(y)}{6}\left[\p^3_\xi\left(\phi(\xi,Y(t,\xi,y))\right)
\right]_{t=T^*(y),\xi=\Xi^*(y)}(\xi-\Xi^*(y))^3\no\\
&& +O\left((t-T^*(y))^2
+(t-T^*(y))|\xi-\Xi^*(y)|^2+|\xi-\Xi^*(y)|^4\right)\no\\
&=&(\phi^*+\psi^*\f{\p_\eta \phi^*}{\p_\xi \phi^*})(t-T^*(y))-(\xi-\Xi^*(y))\no\\
&&+\left(T^*(y)\p^2_{t\xi}\left(\phi(\xi,Y(t,\xi,y))\right)|_{t=T^*(y),\xi=\Xi^*(y)}
+(\p_\xi \phi^*+\p_\eta \phi^*\p_\xi  Y^*)\right)(t-T^*(y))(\xi-\Xi^*(y))\no\\
&&+\f{T^*(y)}{6}\left[\p^3_\xi\left(\phi(\xi,Y(t,\xi,y))\right)
\right]_{t=T^*(y),\xi=\Xi^*(y)}(\xi-\Xi^*(y))^3\no\\
&&+O\left((t-T^*(y))^2+(t-T^*(y))|\xi-\Xi^*(y)|^2+|\xi-\Xi^*(y)|^4\right).
\een
Similarly to \eqref{2.29-2} and \eqref{2.29-3}, we can arrive at
\beq
\label{2.42}c_1(y)\triangleq-\left(T^*(y)\left[\p^2_{t\xi}\left(\phi(\xi,Y(t,\xi,y))\right)\right]_{t=T^*(y),\xi=\Xi^*(y)}
+(\p_\xi \phi^*+\p_\eta \phi^* \p_\xi Y^*)\right)=-\f{1}{\p_\xi \phi(0,0)}+O(\dl)>0
\eeq
and
\beq
\label{2.43}c_2(y)\triangleq\f{T^*(y)}{6}\p^3_\xi\left(\phi(\xi,Y(t,\xi,y))\right)
|_{t=T^*(y),\xi=\Xi^*(y)}=-\f{1+\th_0}{\p_\xi \phi(0,0)}+O(\dl)>0.
\eeq
In this case, \eqref{2.36} can be rewritten as
\ben\label{2.44}
\vG(s,\la,\mu)&\triangleq&-c_1(y)\mu+c_2(y)\mu^3+O(s^2+s\mu^2)-\la=0.
\een
By the implicit function theorem near the point $(s,\mu,\la)=(0,\pm \sqrt{\f{c_1(y)}{c_2(y)}},0)$,
for $s\in[0,\ve)$, $|\la|<\ve$ and $y\in(-\dl,\dl)$, there exist
two real roots $\mu_\pm(s,\la)$ of \eqref{2.44} to fulfill
\beq\label{2.45}
\mu_\pm(s,\la)=\pm\sqrt{\f{c_1(y)}{c_2(y)}}+\f{\la}{2c_1(y)}+O(s+\la^2).
\eeq
Then \eqref{2.39} and \eqref{2.40} are proved.
\end{proof}

\begin{remark}\label{remark 1.1}
Although there are three real roots for the equation \eqref{2.36} of $\xi$, which are denoted by $\xi_+>\xi_c>\xi_-$,
where $\xi_\pm$ are related to $\mu_\pm|_{s=0}=\pm\sqrt{\f{c_2(y)}{c_1(y)}}$
and $\xi_c$ is related to $\mu_c|_{s=0}=0$, the root $\xi_c$ is not utilized in the paper.
\end{remark}

We denote the domain $\O_1$ formed by the cusp surface $\Si_1$:
\beq\label{2.46}
\left(x-x^*(y)-(\phi^*+\psi^*\f{\p_\eta \phi^*}{\p_\xi \phi^*})(t-T^*(y)\right)^2=\ve^2(t-T^*(y))^2,
\eeq
where $t\in [T^*(y),T^*(y)+\ve]$ and $y\in(-\dl,\dl)$.
It is easy to check $\O_1\subset \O_0$ with the cusp domain $\O_0$, where  $\O_0$ is bounded by $\Si_0$.
Next we construct the shock surface $\Si$ in domain $\O_1$.

It follows from \eqref{R-H} that
\beq\label{2.48}
\left\{
\begin{array}{l}
  \p_t \vp(t,y)+\ds\f{[G(u)]}{[u]}\p_y \vp(t,y)=\f{[F(u)]}{[u]}, \\
  \vp(T^*(y),y)=x^*(y)=\Xi^*(y)+T^*(y)\phi(\Xi^*(y),Y^*(y)).
\end{array}
\right.
\eeq

\begin{lemma}\label{lemma 1.3}
Under conditions \eqref{Y-0}-\eqref{Y-1}, for small $\ve>0$, the solution $x=\vp(t,y)$
of \eqref{2.48} exists uniquely in the domain $\{(t,y): 0\le t-T^*(y)<\ve,\ y\in(-\f{\dl}{2},\f{\dl}{2})\}$
and satisfies that\\
(1) $\vp(t,y)$ is a $C^{2}$ function on the variables $(s, y)$, where $s\in[0,\ve)$ and $y\in(-\f{\dl}{2},\f{\dl}{2})$;\\
(2) $\vp(t,y)\in C^{2}\left(\{(t,y): 0\le t-T^*(y)<\ve,\ y\in(-\f{\dl}{2},\f{\dl}{2})\}\right)$ admits the expansion
\beq\label{2.49}
\vp(t,y)=x^*(y)+(\phi^*+\psi^*\f{\p_\eta \phi^*}{\p_\xi \phi^*})(t-T^*(y))+O((t-T^*(y))^2);
\eeq
(3) the entropy condition \eqref{entropy condition} holds.
\end{lemma}
\begin{proof}
At first, we derive the asymptotic properties of $\f{[F(u)]}{[u]}$ and $\f{[G(u)]}{[u]}$ near $\G$.
It follows from direct computation that
\ben\label{2.50}
\f{[F(u)]}{[u]}&=&\f{1}{2}\left[\phi(\xi_+,Y(t,\xi_+,y))
+\phi(\xi_-,Y(t,\xi_-,y)))\right]
+O\left((\xi_+-\Xi^*(y))^2+(\xi_--\Xi^*(y))^2\right)\no\\
&=&\phi^*+\f{1}{2}\left(\p_\xi \phi^*+\p_\eta \phi^*\p_\xi  Y^*\right)
\left((\xi_+-\Xi^*(y))+(\xi_--\Xi^*(y))\right)
+O\left((\xi_+-\Xi^*(y))^2+(\xi_--\Xi^*(y))^2\right)\no\\
&=&\phi^*-\f{1}{T^*(y)}\f{s\la}{2c_1(y)}+O(s^2+s\la^2)
\een
and
\ben\label{2.51}
\f{[G(u)]}{[u]}&=&\psi^*+\f{1}{2}\left(\p_\xi \psi^*+\p_\eta \psi^*\p_\xi  Y^*\right)
\left((\xi_+-\Xi^*(y))+(\xi_--\Xi^*(y))\right)
+O\left((\xi_+-\Xi^*(y))^2+(\xi_--\Xi^*(y))^2\right)\no\\
&=&\psi^*-\f{\p_\xi \psi^*}{T^*(y)\p_\xi \phi^*}\f{s\la}{2c_1(y)}+O(s^2+s\la^2).
\een

In addition, by \eqref{2.16} and \eqref{2.17}, we deduce
\beq
\f{d}{dy}T^*(y)=-\f{T^*(y)\p_\eta H^*\p_y  Y^*}{H^*+T^*(y)\p_\eta H^*\p_t Y^*}.
\eeq
Due to $s=\left(t-T^*(y)\right)^\f{1}{2}$, we then  have
\beq
\p_t s=\f{1}{2s},\quad  \p_y s=-\f{1}{2s}\f{d}{dy}T^*(y)=\f{T^*(y)\p_\eta H^*\p_y  Y^*}{2s(H^*+T^*(y)\p_\eta H^*\p_t  Y^*)}.
\eeq
Thus the equation in \eqref{2.48} becomes
\beq
\f{1}{2s}\left(1-\f{[G(u)]}{[u]}\f{d}{dy}T^*(y)\right)\p_s \vp+\f{[G(u)]}{[u]}\p_y \vp=\f{[F(u)]}{[u]}.
\eeq
Set $\la=\ds\f{\vp-x^*(y)-(\phi^*+\psi^*\f{\p_\eta \phi^*}{\p_\xi \phi^*})(t-T^*(y))}{s^3}$.
Then $\vp=\Xi^*(y)+T^*(y)\phi^*+(\phi^*+\psi^*\f{\p_\eta \phi^*}{\p_\xi \phi^*})s^2+s^3\la$, and we have that
\beq
\p_s\vp=s^3\p_s \la+3s^2\la+2s\left(\phi^*+\f{\p_\eta \phi^*}{\p_\xi \phi^*}\psi^*\right)
\eeq
and
\beq
\p_y\vp=s^3\p_y \la+\f{d}{dy}\left(\Xi^*+T^*(y)\phi^*\right)
+s^2\f{d}{dy}\left(\phi^*+\psi^*\f{\p_\eta \phi^*}{\p_\xi \phi^*}\right).
\eeq
Thus \eqref{2.48} becomes
\beq\label{Y-3}
\left\{
\begin{array}{l}
s C_0(s,\la,y)\p_s\la+s^2C_1(s,\la,y)\p_y\la=C_2(s,\la,y),\\
\la(0,y)=0,
\end{array}
\right.
\eeq
where
\ben
C_0&=&\f{1}{2}\left(1-\f{[G(u)]}{[u]}\f{d}{dy}T^*(y)\right)=\f{1}{2}+O(y),\no\\
C_1&=&\f{[G(u)]}{[u]}=O(y+s^2+s\la),\no\\
C_2&=&-\f{3\la}{2}\left(1-\f{[G(u)]}{[u]}\f{d}{dy}T^*(y)\right)
-s\f{[G(u)]}{[u]}\f{d}{dy}\left(\phi^*+\psi^*\f{\p_\eta \phi^*}{\p_\xi \phi^*}\right)\no\\
&&+\f{1}{s}\left\{\f{[F(u)]}{[u]}-\left(1-\f{[G(u)]}{[u]}\f{d}{dy}T^*(y)\right)
\cdot\left(\phi^*+\f{\p_\eta \phi^*}{\p_\xi \phi^*}\psi^*\right)
-\f{[G(u)]}{[u]}\f{d}{dy}\left(\Xi^*+T^*(y)\phi^*\right)\right\}\no\\
&=&\la\left\{-\f{3}{2}\left(1-\psi^*\f{d}{dy}T^*(y)\right)-\f{1}{2T^*(y)c_1(y)}
-\f{\p_\xi \psi^*}{2T^*(y)\p_\xi \phi^* c_1(y)}\cdot\f{d}{dy}T^*(y)
\left(\phi^*+\psi^*\f{\p_\eta \phi^*}{\p_\xi \phi^*}\right)\right.\no\\
&&\left.+\f{\psi^*_\xi}{2T^*(y)\phi^*_\xi c_1(y)}\cdot\f{d}{dy}
\left(\Xi^*+T^*(y)\phi^*\right)\right\}+O(s+\la^2)\no\\
&&+\f{1}{s}\left\{\phi^*-\left(1-\psi^*\f{d}{dy}T^*(y)\right)\cdot
\left(\phi^*+\psi^*\f{\p_\eta \phi^*}{\p_\xi \phi^*}\right)
-\psi^*\f{d}{dy}\left(\Xi^*+T^*(y)\phi^*\right)\right\}.
\een
Due to $y\equiv Y^*(y)+T^*(y)\psi(\Xi^*(y),Y^*(y))$, direct computation yields
\ben
&&\phi^*-\left(1-\psi^*\f{d}{dy}T^*(y)\right)\cdot
\left(\phi^*+\psi^*\f{\p_\eta \phi^*}{\p_\xi \phi^*}\right)
-\psi^*\f{d}{dy}\left(\Xi^*+T^*(y)\phi^*\right)\no\\
&=&\psi^*\left\{-\f{\p_\eta \phi^*}{\p_\xi \phi^*}+\psi^*\f{\p_\eta \phi^*}{\p_\xi \phi^*}\f{d}{dy}T^*(y)
-\f{d}{dy}\Xi^*(y)-T^*(y)\left[\p_\xi \phi^*\f{d}{dy}\Xi^*(y)+\p_\eta \phi^*\f{d}{dy}Y^*\right]\right\}\no\\
&=&\psi^*\left\{-\f{\p_\eta \phi^*}{\p_\xi \phi^*}+\psi^*\f{\p_\eta \phi^*}{\p_\xi \phi^*}\f{d}{dy}T^*(y)
+T^*(y)\psi^*_\eta\f{d}{dy}\Xi^*(y)-T^*(y)\p_\eta \phi^*\f{d}{dy}Y^*\right\}\no\\
&=&\f{\psi^*\p_\eta \phi^*}{\p_\xi \phi^*}\left\{-1+\psi^*\f{d}{dy}T^*(y)
+T^*\p_\xi \psi^*\f{d}{dy}\Xi^*(y)+T^*(y)\p_\eta \psi^*\f{d}{dy}Y^*+\f{d}{dy}Y^*\right\}.\no\\
&\equiv&0.
\een
Consequently,
\beq
C_2=\left(-\f{3}{2}-\f{1}{2\bar c_1}+O(y)\right)\la+O(s+\la^2),
\eeq
where $\bar c_1=c_1(0)=-\f{1}{\p_\xi(0,0)}$.

We define the characteristics of \eqref{Y-3} starting from the point $(0, \beta)$ with $\beta\in(-\dl,\dl)$ as follows
\beq\label{Y-4}
\left\{
\begin{array}{l}
\f{d}{ds}y(s;\beta)=s\f{C_1}{C_0}(s,\la(s,y(s;\beta)),y(s;\beta)),\\
y(0;\beta)=\beta.
\end{array}
\right.
\eeq
Along this characteristics, set $\La(s;\beta)=\la(s,y(s;\beta))$, we then have
\beq\label{Y-5}
\left\{
\begin{array}{l}
\f{d}{ds}\La(s;\beta)=\f{1}{s}\f{C_2}{C_0}(s,\La(s;\beta),y(s;\beta))\\
\La(0;\beta)=0.
\end{array}
\right.
\eeq
Collecting \eqref{Y-4} and \eqref{Y-5}, we obtain a nonlinear ordinary differential equation
system of $(y(s;\beta), \La(s;\beta))$ as follows
\beq\label{Y-6}
\left\{
\begin{array}{l}
\f{d}{ds}y(s;\beta)=s\f{C_1}{C_0}(s,\La(s;\beta),y(s;\beta)),\\
s\f{d}{ds}\La(s;\beta)=\f{C_2}{C_0}(s,\La(s;\beta),y(s;\beta)),\\
y(0;\beta)=\beta,\ \La(0;\beta)=0,
\end{array}
\right.
\eeq
where $\f{C_1}{C_0}(s,\La,y),\ \f{C_2}{C_0}(s,\La,y)\in C^2$.
By Lemma \ref{4.1} in Appendix,  \eqref{Y-6} has a unique solution $\left(y(s,\beta),\La(s,\beta)\right)$
$\in C^2\left([0,\ve)\times(-\dl,\dl)\right)$,
moreover, $\f{1}{2}\leq\f{\p y}{\p\beta}(s,\beta)\leq \f{3}{2}$ holds for $s\in[0,\ve)$ and small $\ve$.
This derives $\beta=\beta(s,y)\in C^2\left([0,\ve)\times(-\f{\dl}{2},\f{\dl}{2})\right)$ by the implicit
function theorem and then $\La(s,\beta(s,y))\in C^2\left([0,\ve)\times(-\f{\dl}{2},\f{\dl}{2})\right)$.
In addition, it follows from \eqref{Y-5} that $\La(s,\beta(s,y))=O(s)$ holds.
Therefore,
$$
x=\vp(t,y)=s^3\La(s,\beta(s,y))+x^*(y)+s^2(\phi^*+\psi^*\f{\p_\eta \phi^*}{\p_\xi \phi^*})
\in C^2\left(\{(t,y): 0\le t-T^*(y)<\ve,\ y\in(-\f{\dl}{2},\f{\dl}{2})\}\right)
$$
and then the expansion \eqref{2.49} is proved.

At last we verify the entropy condition \eqref{entropy condition}. At the point $P=(t,\vp(t,y),y)$, the normal vector
of $\Si$ is
$(-\p_t \vp,1,-\p_y \vp)$ and the related two characteristic vectors are
$\left(1,\phi(\xi_\pm,Y(t,\xi_\pm,y)),\psi(\xi_\pm,Y(t,\xi_\pm,y)\right)$,
where $\xi_\pm=\xi_\pm(t,\vp(t,y),y)$ (see Figure 5 below). Without loss of generality,
$\xi_-<\xi_+$ is assumed. Since the Rankine-Hugoniot
condition on $\Si$ is
\beq
\left(1,\f{[F(u)]}{[u]},\f{[G(u)]}{[u]}\right)\cdot(-\p_t \vp,1,-\p_y \vp)=0,
\eeq
then for fixed $t$ and $y$, there is a $\bar\xi\in[\xi_-,\xi_+]$ such that
\beq\label{Y-7}
\p_t\vp=\phi(\bar\xi,Y(t,\bar\xi,y))-\psi(\bar\xi,Y(t,\bar\xi,y))\cdot\p_y\vp.
\eeq
Next we prove that $\vF(\xi)\triangleq\phi(\xi,Y(t,\xi,y))-\psi(\xi,Y(t,\xi,y))\cdot\p_y\vp$
is decreasing for $\xi\in[\xi_-,\xi_+]$ when the fixed $t\in(T^*(y),T^*(y)+\ve]$ and $y\in(-\dl,\dl)$.
Indeed, Note that
\beq
\vF'(\xi)=\f{\p_\xi\phi}{1+t\p_\eta\psi}-\f{\p_\xi\psi}{1+t\p_\eta\psi}\cdot\p_y\vp.
\eeq
In addition, by \eqref{2.31} we have
\beq
\p_y\vp(t,y)=-\f{\p_\xi H^*\p_\eta\phi^*+\p_\eta H^*\p_\eta \psi^*}{\p_\xi H^*\p_\xi \phi^*+\p_\eta H^*\p_\xi \psi^*}+O(s)
=-\f{\p_\eta \phi^*}{\p_\xi \phi^*}+O(s).
\eeq
Thus by $H^*=\p_\xi \phi^*+\p_\eta \psi^*=-1+O(y)$, we obtain that for $t\in[T^*(y),T^*(y)+\ve]$ and $y\in(-\dl,\dl)$,
\beq
\vF'(\xi)=\f{\p_\xi \phi^*}{1+T^*(y)\p_\eta \psi^*}
+\f{\p_\xi \psi^*}{1+T^*(y)\p_\eta \psi^*}\cdot\f{\p_\eta \phi^*}{\p_\xi \phi^*}+O(s)=-\f{1}{1+T^*(y)\p_\eta \psi^*}+O(s+|y|)<0,
\eeq
This means that $\vF(\xi)$
is decreasing on the variable $\xi\in[\xi_-,\xi_+]$ and then $\vF(\xi_+)<\vF(\bar\xi)<\vF(\xi_-)$ holds.
This together with \eqref{Y-7} yields that on $\Sigma$
\beq
\phi(\xi_+,Y(t,\xi_+,y))-\psi(\xi_+,Y(t,\xi_+,y))\cdot\p_y\vp
<\p_t\vp<\phi(\xi_-,Y(t,\xi_-,y))-\psi(\xi_-,Y(t,\xi_-,y))\cdot\p_y\vp,
\eeq
which derives the entropy condition \eqref{entropy condition}.
\end{proof}

\begin{figure}[h]
\centering
\includegraphics[scale=0.50]{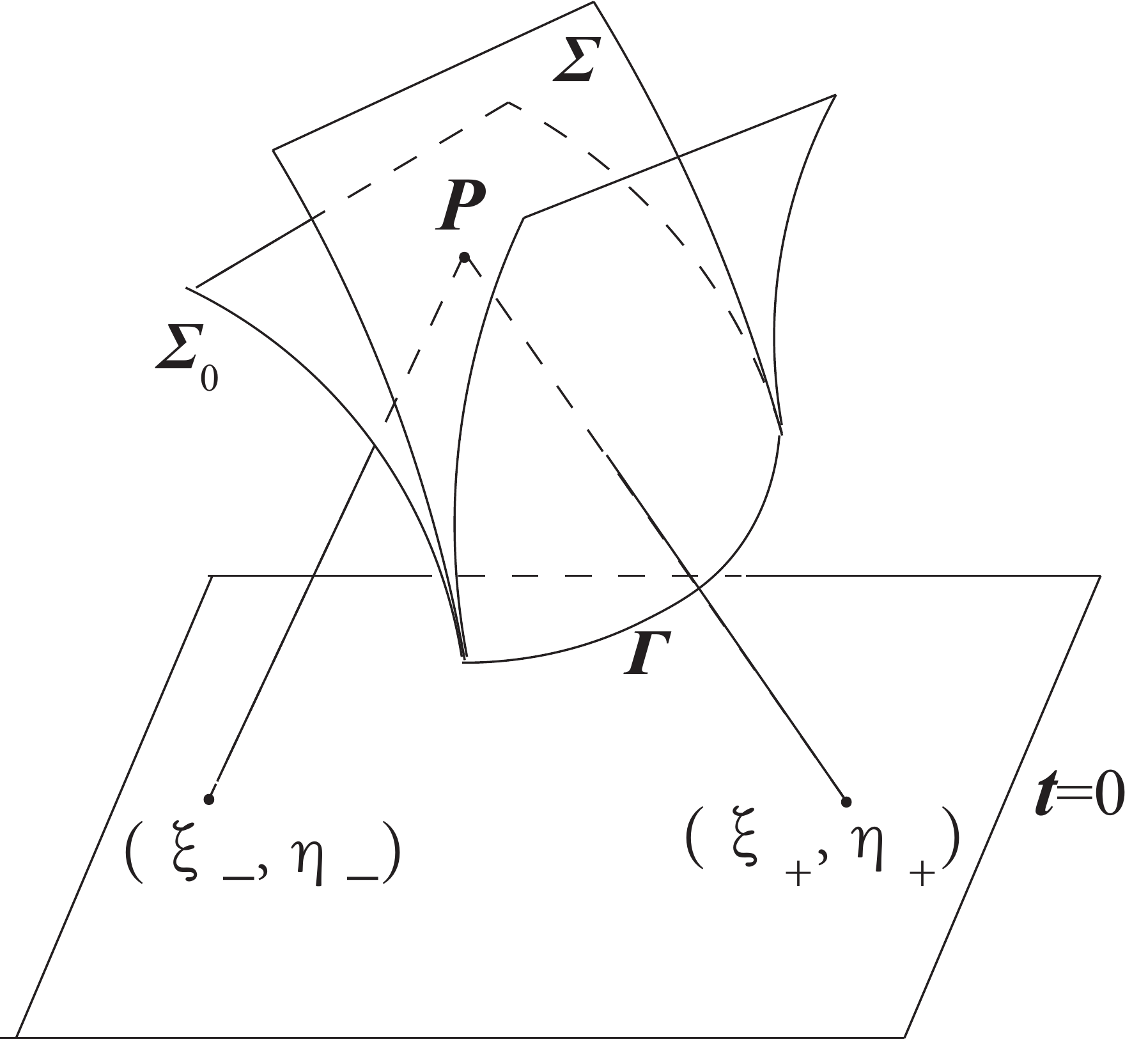}
\end{figure}
\centerline{\bf Figure 5.}

\section{Behavior of the solution $u$ near $\G$ and proof of Theorem 1.1.}
It is easy to see that the map $(\xi,\eta)\to (x(t,\xi,\eta),y(t,\xi,\eta))$ in \eqref{2.5}
is invertible outside the cusp domain $\O_0$.
Moreover, for $t\in[T^*(y),T^*(y)+\ve)$ and $y\in(-\dl,\dl)$,
there exists a positive constant $\al_0(\dl)$
depending on $\dl$ such that for $x>x^*(y)+(\phi^*+\psi^*\f{\p_\eta \phi^*}{\p_\xi \phi^*})(t-T^*(y))+\left(\f{2}{3\p_\xi \phi(0,0)\sqrt{3+3\th_0}}+\al_0(\dl)\right)(t-T^*(y))^\f{3}{2}$,
there exists a $\xi_+(t,x,y)$
satisfying \eqref{2.5}; and for $x<x^*(y)+(\phi^*+\psi^*\f{\p_\eta \phi^*}{\p_\xi \phi^*})(t-T^*(y))+\left(-\f{2}{3\p_\xi \phi(0,0)\sqrt{3+3\th_0}}+\al_0(\dl)\right)(t-T^*(y))^\f{3}{2}$,
there exists a $\xi_-(t,x,y)$
satisfying \eqref{2.5}. From \eqref{2.36}, we start to  improve Lemma 2.2 so that the better
asymptotic behaviors of $\xi_\pm(t,x,y)$ are obtained.

\begin{lemma}\label{lemma 1.4}
For each number $c>-\f{2}{3\p_\xi \phi(0,0)\sqrt{3+3\th_0}}+\al_0(\dl)$,
there exist a constant $\ve(c)>0$
such that for $s\in[0,\ve(c))$ and $|\la-c|<\ve(c)$, $(s,\la)\to \xi_+$ satisfies
\beq\label{3.1}
\xi_+(t,x,y)=\Xi^*(y)+s\left(\mu_c(y)+\f{\la-c}{-c_1(y)+3c_2\mu^2_c(y)}\right)+O(s^2+s|\la-c|^2),
\eeq
and for $s\in[0,\ve(c))$ and $|\la+c|<\ve(c)$, $(s,\la)\to \xi_-$ satisfies
\beq\label{3.2}
\xi_-(t,x,y)=\Xi^*(y)+s\left(-\mu_c(y)+\f{\la+c}{-c_1(y)+3c_2\mu^2_c(y)}\right)+O(s^2+s|\la+c|^2),
\eeq
where $\mu_c(y)$ is the solution of the following algebraic equation for $y\in(-\dl,\dl)$
\beq\label{3.3}
\vG_0(\mu;y)=-c_1(y)\mu+c_2(y)\mu^3=c.
\eeq
\end{lemma}
\begin{proof}
Note that for $c>-\f{2}{3\p_\xi \phi(0,0)\sqrt{3+3\th_0}}+\al_0(\dl)$,
there is a real root $\mu_c(y)>\f{1}{\sqrt{3+3\th_0}}$ for \eqref{3.3}.
On the other hand,  by \eqref{2.44}, one can arrive at
\ben
\label{3.4}\p_\mu\vG(0,\pm c,\pm \mu_c(y))&=&-c_1(y)+3c_2(y)\mu_c^2(y)>0,\\
\label{3.5}\p_\la\vG(0,\pm c,\pm \mu_c(y))&=&-1.
\een
Similarly  to the proof of Lemma \ref{lemma 1.2}
by implicit function theorem, we can obtain \eqref{3.1} and \eqref{3.2}.
\end{proof}

\begin{remark}\label{remark 1.3}
We point out that if $c\geq0$, then $\mu_c(y)\geq\sqrt{\f{c_1(y)}{c_2(y)}}
=\f{1}{\sqrt{1+\th_0}}+O(\dl)>0$ and $-c_1(y)+3c_2(y)\mu_c^2(y)\geq c_0>0$ hold for some constant $c_0>0$
when $\dl>0$ is sufficiently small.
\end{remark}

To study the asymptotic behaviors of $\xi_\pm$ near
$x=x^*(y)$,
we now take the following transformation
\beq\label{3.6}
\zeta=\left(x-x^*(y)-(\phi^*+\psi^*\f{\p_\eta \phi^*}{\p_\xi \phi^*})(t-T^*(y))\right)^\f{1}{3},\
\eta=\f{t-T^*(y)}{\zeta^{2}},\ \nu=\f{\xi-\Xi^*(y)}{\zeta}.
\eeq
By dividing the factor  $\zeta^{3}$ for \eqref{2.36},  we have
\beq\label{3.7}
H(\eta,\zeta,\nu)\triangleq
-c_1(y)\eta\nu+c_2(y)\nu^{3}+O(\eta^2\zeta+\zeta\nu^2\eta+\zeta\nu^4)-1=0.
\eeq
Note that for $|\eta|$ small enough, there is a unique real root $\nu$ for \eqref{3.7}. Furthermore, we have

\begin{lemma}\label{lemma 1.5}
There exists a constant $\ve>0$ such that for $(\eta,\zeta)\in\{ |\eta|\le \ve,\ |\zeta|<\ve\}$,
the expansion of $\xi(t,x,y)$ on $(\eta,\zeta)$ admits
\beq
\label{3.8}\xi(t,x,y)
=\Xi^*(y)+\zeta\left(\f{1}{\sqrt[3]{c_2(y)}}
+\f{c_1(y)}{3\left(c_2(y)\right)^\f{2}{3}}\eta\right)
+O(\zeta^2+\zeta\eta^2),
\eeq
where $y\in(-\dl,\dl)$.
\end{lemma}
\begin{proof}
For $\eta=0$ and $\zeta=0$, we have $\nu=\f{1}{\sqrt[3]{c_2(y)}}$. Due to
\ben
\p_\nu H(0,0,\f{1}{\sqrt[3]{c_2(y)}})&=&3\sqrt[3]{c_2(y)}>0,\\
\p_\eta H(0,0,\f{1}{\sqrt[3]{c_2(y)}})&=&-\f{c_1(y)}{\sqrt[3]{c_2(y)}},
\een
we then have that for $|\eta|\le \ve$ and $|\zeta|<\ve$ with $\ve$ being small,
\beq
\label{3.9}
\nu=\f{1}{\sqrt[3]{c_2(y)}}+\f{c_1(y)}{3\left(c_2(y)\right)^\f{2}{3}}\eta+O(\zeta+\eta^2),
\eeq
which derives \eqref{3.8}.
\end{proof}

Next we consider the behavior of $\xi(t,x,y)$ near $\G$ for $t<T^*(y)$.
Without of confusion to \eqref{2.37}, we still denote the transformation
\beq\label{3.10}
s=(T^*(y)-t)^\f{1}{2},\ \mu=\f{\xi-\Xi^*(y)}{s},\
\la=\f{x-x^*(y)-(\phi^*+T^*(y)\p_\eta \phi^*\p_t Y^*)(t-T^*(y))}{s^3}.
\eeq
Similarly to \eqref{2.44}, \eqref{2.36} becomes
\beq
\label{3.11}
\tilde\vG(s,\la,\mu)\triangleq c_1(y)\mu+c_2(y)\mu^3+O(s^2+s\mu^2)-\la=0.
\eeq
Then we have

\begin{lemma}\label{lemma 1.6}
For each $c\in\bR$,
there exists a positive constant $\ve(c)$
such that for $s\in[0,\ve(c))$ and $|\la-c|<\ve(c)$, $(s,\la)\to \xi$ satisfies
\beq\label{3.12}
\xi(t,x,y)=\Xi^*(y)+s\left(\tilde\mu_c(y)+\f{\la-c}{c_1(y)+3c_2\tilde\mu^2_c(y)}\right)+O(s^2+s|\la-c|^2),
\eeq
where $y\in(-\dl,\dl)$, and $\tilde\mu_c(y)$ is the real root of
\beq\label{3.13}
\tilde\vG_0(\mu;y)=c_1(y)\mu+c_2(y)\mu^3=c.
\eeq
\end{lemma}
\begin{proof}
It is easy to check that for $c\in\bR$,
there is a real root $\tilde\mu_c(y)$ of \eqref{3.13}.
On the other hand, by \eqref{3.11}, one has that
\ben
\label{3.14}\p_\mu\tilde\vG(0,\pm c,\tilde\mu_c(y))&=&c_1(y)+3c_2(y)\tilde\mu_c^2(y)>0,\\
\label{3.15}\p_\la\tilde\vG(0,\pm c,\tilde\mu_c(y))&=&-1.
\een
Similarly to the proof of  Lemma \ref{lemma 1.2}
by implicit function theorem, we then obtain \eqref{3.12}.
\end{proof}
\begin{remark}
It is obvious that for $y\in(-\dl,\dl)$, $c_1(y)+3c_2(y)\tilde\mu_c^2(y)\geq -\f{1}{\p_\xi \phi(0,0)}+O(\dl)>0$ holds.
\end{remark}

{\bf Proof of Theorem 1.1 (3).}
\begin{proof}
We now establish the behavior of the solution $u$ and its derivatives near $\G$.
Denote $B_\rho=\{(t,x,y):|x-x^*(y)|<\rho,\ |t-T^*(y)|<\rho,\ |y|<\rho\}$
for some positive constant $\rho$ defined below.
Let $\ve>0$ be the constant obtained in Lemma \ref{lemma 1.5}.
Set
\beq
\label{3.16}\O_{x}=B\cap\{(t,x,y):\ |x-x^{**}(t,y)|^\f{1}{3}<\ve,\ |t-T^*(y)|<\ve |x-x^*(y)|^\f{2}{3},\ y\in(-\dl,\dl)\},
\eeq
where $x^{**}(t,y)=x^*(y)+(\phi^*+\psi^*\f{\p_\eta \phi^*}{\p_\xi \phi^*})(t-T^*(y))$.
In addition, let
\ben
\label{3.17}\O_0&=&B\cap\{(t,x,y):\ t<T^*(y),\
|x-x^{**}(t,y)|<\f{2}{\ve^\f{3}{2}}(T^*(y)-t)^\f{3}{2},\ y\in(-\dl,\dl)\},\\
\label{3.18}\O_{t,0}&=&B\cap\{(t,x,y):\ 0<(t-T^*(y))^\f{1}{2}<\ve,\
|x-x^{**}(t,y)|<\ve(t-T^*(y))^\f{3}{2},\ y\in(-\dl,\dl)\},\\
\label{3.19}\O_{t,+}&=&B\cap\{(t,x,y):\ t>T^*(y),\
\f{\ve}{2}<\f{x-x^{**}(t,y)}{(T^*(y)-t)^\f{3}{2}}<\f{2}{\ve^\f{3}{2}},\ y\in(-\dl,\dl)\},\\
\label{3.20}\O_{t,-}&=&B\cap\{(t,x,y):\ t>T^*(y),\
-\f{2}{\ve^\f{3}{2}}<\f{x-x^{**}(t,y)}{(T^*(y)-t)^\f{3}{2}}<-\f{\ve}{2},\ y\in(-\dl,\dl)\}.
\een
where $\ve$ is the same as in \eqref{3.16}, and $\ve$ satisfies the requirements in
Lemma \ref{lemma 1.2}. By Heine-Borel property of compactness, one can
choose $\{c_{j,\pm},\ve_{j,\pm}=\ve_{j,\pm}(c_{j,\pm})\}_{j=1}^n$ and $\{c_{j,0},
\ve_{j,0}=\ve_{j,0}(c_{j,0})\}_{j=1}^n$ such that
\beq
\label{3.21}
\O_{t,+}\subset\cup_{j=1}^n\O_{t,+}^j,\ \O_{t,-}\subset\cup_{j=1}^n\O_{t,-}^j,\ \O_{0}\subset\cup_{j=1}^n\O_{0}^j,
\eeq
where
\ben
\label{3.22} \O_{t,+}^j&=&\{(t,x,y):\ 0<(t-T^*(y))^{\f{1}{2}}<\ve_{j,+},\ c_{j,+}-\ve_{j,+}<\f{x-x^{**}(t,y)}{(t-T^*(y))^\f{3}{2}}<c_{j,+}+\ve_{j,+}\},\\
\label{3.23} \O_{t,-}^j&=&\{(t,x,y):\ 0<(t-T^*(y))^{\f{1}{2}}<\ve_{j,-},\ c_{j,-}-\ve_{j,-}<\f{x-x^{**}(t,y)}{(t-T^*(y))^\f{3}{2}}<c_{j,-}+\ve_{j,-}\},\\
\label{3.24} \O_{0}^j&=&\{(t,x,y):\ 0<(T^*(y)-t)^{\f{1}{2}}<\ve_{j,0},\ c_{j,0}
-\ve_{j,0}<\f{x-x^{**}(t,y)}{(T^*(y)-t)^\f{3}{2}}<c_{j,0}+\ve_{j,0}\},
\een
and these domains satisfy the corresponding properties in Lemma \ref{lemma 1.4} and \ref{lemma 1.6}.

We take $\rho>0$ sufficiently small such that
\beq\label{3.25}
B_\rho=\O_{x,+}\cup\O_{x,-}\cup\O_{t,+}\cup\O_{t,-}\cup\O_{0}.
\eeq

In order to derive the behaviors of $u$ and its derivatives near $\G$, it suffices to
only consider them in the domains $\O_{x,+}$, $\O_{t,+}^j$ and  $\O_{0}^j$
since the other cases can be treated analogously.

It follows from direct computation that for fixed $y\in(-\dl,\dl)$,
\ben\label{3.26}
|u(t,x,y)-u(T^*(y),x^*(y),y)|&=&|u_0(\xi(t,x,y),Y(t,x,y))-u_0(\Xi^*(y),Y(T^*(y),\Xi^*(y),y))|\no\\
&\lesssim&|\xi(t,x,y)-\Xi^*(y)|+|t-T^*(y)|,
\een
here we have used the boundedness of the derivatives of $u_0$ and the variable $Y$.
Thus, for $(t,x,y)\in\O_{x,+}$,
\beq
\label{3.27}
|u(t,x,y)-u(T^*(y),x^*(y),y)|\lesssim \varsigma^\f{1}{3}+|t-T^*(y)|\lesssim
|\varsigma|^\f{1}{3},
\eeq
where and below $\varsigma=x-x^{**}(t,y)$;
for $(t,x)\in\O_{t,+}^j$,
\beq
\label{3.28}
|u(t,x,y)-u(T^*(y),x^*(y),y)|\lesssim  |\xi_+(t,x,y)-\Xi^*(y)|+(t-T^*(y))\lesssim  (t-T^*(y))^\f{1}{2};
\eeq
and for $(t,x)\in\O_{0}^j$,
\beq
\label{3.29}
|u(t,x,y)-u(T^*(y),x^*(y),y)|=|u_0(y(t,x))|\lesssim |\xi(t,x,y)-\Xi^*(y)|+(T^*(y)-t)\lesssim (T^*(y)-t)^\f{1}{2}.
\eeq
Collecting \eqref{3.27}, \eqref{3.28} and \eqref{3.29}, then \eqref{e0} is obtained.

Next we consider the estimate \eqref{e1} on the derivatives of $u$. By \eqref{2.5} and direct computation, it follows that
\beq\label{3.30}
\left(
  \begin{array}{ccc}
    \p_t \xi & \p_x \xi & \p_y \xi \\
    \p_t \eta & \p_x \eta & \p_y  \eta\\
  \end{array}
\right)
=-\f{1}{1+tH}\left(
  \begin{array}{ccc}
    \phi+t(\phi\p_\eta \psi-\p_\eta \phi\psi) & -(1+t\p_\eta \psi) & t\p_\eta \psi \\
    \psi+t(\p_\xi \phi\psi-\phi\p_\xi \psi) & t\p_\xi \psi & -(1+t\p_\xi \phi) \\
  \end{array}
\right).
\eeq
Then
\ben
\label{3.31}
\p_t u&=&\p_\xi u_0\p_t \xi+\p_\eta u_0\p_t\eta
=-\f{1}{1+tH}\left(\phi\p_\xi u_0+\psi\p_\eta u_0\right),\\
\p_x u&=&\f{\p_\xi u_0}{1+tH},\\
\p_y u&=&\f{\p_\eta u_0}{1+tH}.
\een
Near $\G$, we have
\ben\label{3.32}
&&1+tH(\xi,Y(t,\xi,y))\no\\
&=&(H^*+\p_\eta H^*\p_t Y^*)(t-T^*(y))+\f{1}{2}\left.T^*(y)\p_\xi^2
\left[H(\xi,Y(t,\xi,y))\right]\right|_{t=T^*(y),\ x=x^*(y)}(\xi-\Xi^*(y))^2\no\\
&&+O\left((t-T^*(y))^2+|t-T^*(y)||\xi-\Xi^*(y)|\right)\no\\
&=&\left(-1+O(\dl)\right)(t-T^*(y))+(3+3\th_0+O(\dl))(\xi-\Xi^*(y))^2\no\\
&&+O\left((t-T^*(y))^2+|t-T^*(y)||\xi-\Xi^*(y)|\right).
\een
For $(t,x)\in\O_{t,+}^j$, by \eqref{3.1} in Lemma \ref{lemma 1.4} we have
\ben\label{3.33}
&&1+tH(\xi,Y(t,\xi,y))\no\\
&=&\left(-1
+(3+3\th_0)\mu_{c_{j,+}}^2(y)+O(\dl)\right)(t-T^*(y))+O((t-T^*(y))^2+|t-T^*(y)||\xi-\Xi^*(y)|)\no\\
&\gtrsim&(t-T^*(y))\no\\
&\gtrsim&|t-T^*(y)|+|\varsigma|^\f{2}{3},
\een
where the fact of $-1
+(3+3\th_0)\mu_{c_{j,+}}^2(y)+O(\dl)>0$ in Remark \ref{remark 1.3} has been used.

For $(t,x)\in\O_{x}$, by \eqref{3.8} in Lemma \ref{lemma 1.5} we have
\ben\label{3.34}
&&1+tH(\xi,Y(t,\xi,y))\no\\
&=&\left(-1+O(\dl)\right)(t-T^*(y))
+(3+3\th_0+O(\dl))\f{\varsigma^\f{2}{3}}{\left(c_2(y)\right)^\f{2}{3}}
+O((t-T^*(y))^2+|t-T^*(y)||\varsigma|^\f{1}{3})\no\\
&\gtrsim&\varsigma^\f{2}{3}\no\\
&\gtrsim&|t-T^*(y)|+\varsigma^\f{2}{3}.
\een

For $(t,x)\in\O_{0,+}^j$, by \eqref{3.12} in Lemma \ref{lemma 1.6} we arrive at
\ben\label{3.35}
&&1+tH(\xi,Y(t,\xi,y))\no\\
&=&\left(-1
+(3+3\th_0)\tilde\mu_{c_{j,0}}^2(y)+O(\dl)\right)(t-T^*(y))+O((t-T^*(y))^2+|t-T^*(y)||\xi-\Xi^*(y)|)\no\\
&\gtrsim&(T^*(y)-t)\no\\
&\gtrsim&|t-T^*(y)|+|\varsigma|^\f{2}{3},
\een
where the fact of $1+(3+3\th_0)\tilde\mu_{c_{j,0}}^{2}>0$ has been used.

Therefore, $1+tH\gtrsim|t-T^*(y)|+|\varsigma|^\f{2}{3}$ holds for $(t,x,y)\in B_\rho$.
On the other hand for fixed $y\in(-\dl,\dl)$,
we denote the tangent derivative along the $\Si$ on $\G$ by
\ben\label{3.36}
&&\p_{T} u(t,x,y)\triangleq \p_t u+\left(\phi^*+T^*(y)\p_\eta \phi^*\p_t Y^*\right)\p_x u\no\\
&=&-\f{1}{1+tH}\left[\left(\phi\p_\xi u_0+\psi\p_\eta u_0\right)
-\left(\phi^*+\psi^*\f{\p_\eta \phi^*}{\p_\xi \phi^*}\right)\p_\xi u_0\right]\no\\
&=&-\f{1}{1+tH}\left[\left(\phi-\phi^*\right)
+\left(\psi\f{\p_\eta \phi}{\p_\xi \phi}-\psi^*\f{\p_\eta \phi^*}{\p_\xi \phi^*}\right)\right]\p_\xi u_0,
\een
here we have used the fact of  $\p_\xi \phi\p_\eta u_0=\p_\eta \phi\p_\xi u_0$. Therefore similarly  to
the proof of \eqref{3.26}, we have
\beq\label{3.37}
|\p_{T} u(t,x,y)|\lesssim \left(|t-T^*(y)|^{\f{1}{2}}+|\varsigma|^{\f{1}{3}}\right)^{-1},
\eeq
and then \eqref{eT} is proved.
\end{proof}

\vskip 0.5 true cm

{\bf Proof of Theorem 1.1.}

\begin{proof}
Theorem 1.1 (1) and (2) come from Lemma 2.3 directly. Theorem 1.1 (3) has been obtained.
\end{proof}

\section{Appendix}

In this section, we study problem \eqref{Y-6}
\ben\label{4.1}
\left\{
\begin{array}{l}
\f{d}{ds}y(s;\beta)=sP(s,\La(s;\beta),y(s;\beta)),\\
s\f{d}{ds}\La(s;\beta)=Q(s,\La(s;\beta),y(s;\beta)),\\
y(0;\beta)=\beta,\ \La(0;\beta)=0,
\end{array}
\right.
\een
where $|\beta|\le\dl$ with $\dl>0$ being small, $P,\ Q\in C^2$ satisfy
\beq\label{4.2}
|P(s,\La,y)|\le M|y+s^2+s\La|,\ \text{$Q(s,\La,y)=-\al \La+\tilde Q(s,\La,y)$ with $|\tilde Q(s,\La,y)|\le M|s+y\La+\La^2|$},
\eeq
and the constants $M>1$ and $\al\ge2$.

\begin{lemma}\label{lemma 4.1}
For small $\dl$ and $\ve>0$, \eqref{4.1} with assumption \eqref{4.2} admits a unique local solution
\beq\label{4.3}
(y(s;\beta),\La(s;\beta))\in C^2\left([0,\ve]\times[-\dl,\dl]\right).
\eeq
\end{lemma}
\begin{proof}
Taking the following iterative scheme
\ben\label{4.4}
\left\{
\begin{array}{l}
y_{k}(s;\beta)=\beta+\int_0^s\th P(\th,\La_{k-1}(\th),y_{k-1}(\th))~d\th,\\
\La_{k}(s;\beta)=s^{-\al}\int_0^s \th^{\al-1}\tilde Q(\th,\La_{k-1}(\th),y_{k-1}(\th))~d\th,\\
y_0\equiv\beta,\ \La_0\equiv0,
\end{array}
\right.
\een
where $(y,\La)\in S\triangleq\{(y,\La)\in C([0,\ve]): |y|\le 2\dl, |\La|\le M s\}$.

If $(y_{k-1},\La_{k-1})\in S$, we then have that for small $\dl$ and $\ve$,
\ben\label{4.5}
|y_k|&\le&\dl+\ve\|P\|_{L^\i([0,\ve]\times S)}\le \dl+\ve M(2\dl+\ve^2+M\ve^2)\le 2\dl,\\
|\La_k|&\le& s^{-\al}\int_0^s \th^{\al}\|\f{\tilde Q}{\th}\|_\i~d\th\le\f{s}{\al+1}\|\f{\tilde Q}{s}\|_{L^\i([0,\ve]\times S)}
\le \f{Ms}{\al+1}(1+2M\dl+M^2\ve)\le Ms.
\een
In addition,
\ben\label{4.6}
|y_k-y_{k-1}|&\le&\int_0^{s}\left|\th\left( P(\th,\La_{k-1}(\th),y_{k-1}(\th))-P(\th,\La_{k-2}(\th),y_{k-2}(\th))\right)\right|~d\th\no\\
&\le& C\int_0^s\th\left(\left|y_{k-1}(\th)-y_{k-2}(\th)\right|
+\left|\La_{k-1}(\th)-\La_{k-2}(\th)\right|\right)~d\th\no\\
&\le& C\ve\left(\|y_{k-1}-y_{k-2}\|_{L^\i[0,\ve]}+\|\La_{k-1}-\La_{k-2}\|_{L^\i[0,\ve]}\right),\\
|\La_k-\La_{k-1}|&\le& s^{-\al}\int_0^s \th^{\al-1}\left|\tilde Q(\th,\La_{k-1}(\th),y_{k-1}(\th))-\tilde Q (\th,\La_{k-2}(\th),y_{k-2}(\th))\right|~d\th\no\\
&\le& Cs^{-\al}\int_0^s \th^{\al-1}\left(2M\th\left|y_{k-1}(\th)-y_{k-2}(\th)\right|+
4\dl\left|\La_{k-1}(\th)-\La_{k-2}(\th)\right|\right.\no\\
&&\left.+2M\th\left|\La_{k-1}(\th)-\La_{k-2}(\th)\right|
\right)~d\th\no\\
&\le& C\left[\ve\left(\|y_{k-1}-y_{k-2}\|_{L^\i[0,\ve]}+\|\La_{k-1}-\La_{k-2}\|_{L^\i[0,\ve]}\right)
+\dl\|\La_{k-1}-\La_{k-2}\|_{L^\i[0,\ve]}\right],
\een
where $C$ is a positive constant independent of $\dl$ and $\ve$. Thus for small $\dl$ and $\ve$,
$\left\{(y_k,\La_k)\right\}$ converges uniformly to some functions $(y(s),\La(s))\in C[0,\ve]$ satisfying \eqref{4.4},
moreover $(y(s),\La(s))\in S$. Furthermore, by \eqref{4.2} and \eqref{4.4}, $y(s)$ and $\La(s)\in C^2[0,\ve]$ hold.
\end{proof}


\begin{thebibliography}{llll}




\bibitem{A1}  S. Alinhac, \textit{Blowup of small data solutions for a class of quasilinear wave equations in two space dimensions}.
Ann. of Math. 149 (1999), no. 1, 97-127.


\bibitem{B-1} T. Buckmaster, S. Shkoller, V. Vicol, \textit{Formation of  shocks for 2D isentropic compressible Euler},
arXiv: 1907.03784, 8 Jul 2019.



\bibitem{B-2} T. Buckmaster, S. Shkoller, V. Vicol, \textit{Formation of point shocks for 3D compressible Euler},
arXiv: 1912.04429, 22 Jun 2020.


\bibitem{BSV-3} T. Buckmaster, S. Shkoller, V. Vicol, \textit{Shock formation and vorticity creation for 3d Euler},
arXiv:2006.14789 (2020)



\bibitem{Chen-Xin-Yin} Chen Shuxing, Xin Zhouping, Yin Huicheng, \textit{Formation and construction of
shock wave for quasilinear hyperbolic system and its application to inviscid compressible flow}.
The Institute of Mathematical Sciences at CUHK, 2010, Research Reports: 2000-10(069)


\bibitem{Chen-Dong} Chen Shuxing, Dong Liming, \textit{Formation of shock for the p-system with general smooth initial
data}. Sci. in China, Ser. A, 44 (2001), no. 9, 1139-1147.




\bibitem{C1} D. Christodoulou, \textit{The formation of shocks in 3-dimensional fluids.}
 EMS Monographs in Mathematics. European Mathematical Society (EMS), Z\"urich, 2007.

\bibitem{C2} D. Christodoulou, \textit{ The shock development problem. EMS Monographs in Mathematics.}
 European Mathematical Society (EMS), Z\"urich, 2019. ix+920 pp.

\bibitem{C-L} D. Christodoulou, A. Lisibach, \textit{Shock development in spherical symmetry.}
 Ann. PDE 2 (2016), no. 1, Art. 3, 246 pp.


\bibitem{CM} D. Christodoulou, Miao Shuang, \textit{Compressible flow and Euler's equations},
Surveys of Modern Mathematics, 9. International Press, Somerville, MA; Higher
Education Press, Beijing, 2014.

\bibitem{0-Speck} G. Holzegel, S. Klainerman, J. Speck, Wong Willie Wai-Yeung, \textit{Small-data shock formation in solutions
to 3D quasilinear wave equations: an overview.} J. Hyperbolic Differ. Equ. 13 (2016), no. 1, 1-105.

\bibitem{Hor} L. H\"ormander, \textit{Lectures on nonlinear hyperbolic differential equations}, Mathematics
and Applications 26, Springer-Verlag, Berlin, 1997.

\bibitem{Kong}   Kong Dexing, \textit{Formation and propagation of singularities for
$2\times 2$ quasilinear hyperbolic systems}. Trans. Amer. Math. Soc. 354 (2002), no. 8, 3155-3179.

\bibitem{Le94}    M.P. Lebaud, \textit{Description de la formation d'un choc dans le
$p-$syst\`{e}me}. J. Math. Pures Appl. (9) 73 (1994), no. 6, 523-565.

\bibitem{LS} J. Luk, J. Speck, \textit{Shock formation in solutions to the 2D compressible Euler equations
in the presence of non-zero vorticity}. Invent. Math. 214 (2018), no. 1, 1-169.




\bibitem{Majda-1} A. Majda, \textit{The stability of multidimensional shock fronts}.
 Mem. Amer. Math. Soc. 41 (1983), no. 275, iv+95 pp.

\bibitem{Majda-2} A. Majda, \textit{The existence of multidimensional shock fronts}.
 Mem. Amer. Math. Soc. 43 (1983), no. 281, v+93 pp

\bibitem{Majda-3} A. Majda, \textit{Compressible fluid flow and systems of conservation laws in several space variables}.
Applied Mathematical Sciences, \textbf{53}, Springer-Verlag, New York, 1984.






\bibitem{Merle} F. Merle, P. Rapha\"el, I. Rodniaski, J. Szefel, \textit{On the
implosion of a three dimensional compressible fluid}, arXiv: 1912.11009, 13 Jun 2020.


\bibitem{G.M} G. M\'etivier, \textit{Stability of multidimensional shocks. Advances in the theory of shock waves},
25-103, Progr. Nonlinear Differential Equations Appl., 47, Birkh\"auser Boston, Boston, MA, 2001.


\bibitem{MY} Miao Shuang, Yu Pin, \textit{On the formation of shocks for quasilinear wave equations}.
Invent. Math. 207 (2017), no. 2, 697-831.



\bibitem{Smo}  J. A. Smoller, \textit{Shock waves and reaction-diffusion equations}, Berlin-Heiderberg-New
York, Springer-Verlag, New York, 1984.



\bibitem{S2} J. Speck, \textit{Shock formation for 2D quasilinear wave systems featuring multiple speeds: blowup
for the fastest wave, with non-trivial interactions up to the singularity.}
 Ann. PDE 4 (2018), no. 1, Art. 6, 131 pp.



\bibitem{Yin1} Yin Huicheng,  \textit{Formation and construction of a shock wave for 3-D compressible Euler equations
with the spherical initial data.} Nagoya Math. J. 175 (2004), 125-164.

\bibitem{Y-Z} Yin Huicheng,  Zhu Lu,  \textit{
The shock formation and optimal regularities of the resulting shock curves for 1-D scalar conservation laws},
arXiv:2103.07837, 13 March, 2021.




\end{thebibliography}
\end{document}